\documentclass[reqno]{amsart}
\usepackage{amsmath,amsthm,amsfonts,amssymb}
\RequirePackage{graphicx}
\numberwithin{equation}{section}
\theoremstyle{plain}

\newtheorem{theorem}{Theorem}

\newtheorem{proposition}[theorem]{Proposition}

\newtheorem{remark}[theorem]{Remark}

\begin{document}

\title[Phase transition curve in ERGM]{On the phase transition curve in a directed 
exponential random graph model}

\author{David Aristoff}
\address
{Department of Mathematics \newline
\indent Colorado State University \newline
\indent 1874 Campus Delivery \newline
\indent Fort Collins, CO-80523 \newline
\indent United States of America}
\email{
aristoff@math.colostate.edu}

\author{Lingjiong Zhu}
\address
{Department of Mathematics \newline
\indent Florida State University \newline
\indent 1017 Academic Way \newline
\indent Tallahassee, FL-32306 \newline
\indent United States of America}
\email{
zhu@math.fsu.edu}

\date{26 October 2017.}


\begin{abstract}
We consider a family of directed exponential random graph models
parametrized by edges and outward stars. Much 
of the important  
statistical content of such models is given by the 
normalization constant of the models, 
and in particular, an appropriately 
scaled limit of the normalization, 
which is called the {\em free energy}.  
We derive precise asymptotics for 
the normalization constant for {\em finite} graphs.  
We use this to derive a formula 
for the free energy. 
The limit is analytic everywhere except along a curve corresponding 
to a first order phase 
transition. We 
examine unusual 
behavior of the model along the phase transition curve.
\end{abstract}

\maketitle

\section{Introduction}

Probabilistic ensembles with one or more adjustable parameters
are often used to model complex networks, including social networks, 
biological networks, the Internet, etc.; see 
Fienberg \cite{FienbergI,FienbergII}, 
Lov\'{a}sz \cite{Lovasz2009}
and Newman \cite{Newman}. 
One of the standard complex network models 
are the {\em exponential random graph models} (ERGMs), originally
studied by Besag \cite{Besag}. 
We refer to Snijders et al. \cite{Snijders}, Rinaldo et al. \cite{Rinaldo} and
Wasserman and Faust \cite{Wasserman} for history and 
a review of recent developments.

Much of the statistical content of such models can be described by the probability normalization. 
An appropriately scaled limit of the 
normalization, which is called the {\em free energy}, is 
useful for understanding properties of 
large graphs sampled from ERGMs. The word free 
energy comes from an analogous quantity 
in statistical physics; see Section~\ref{1.2} below. 
In that setting, the free energy 
is used to draw phase diagrams corresponding 
(for example) to the familiar fluid, liquid and solid phases of matter~\cite{FisherRadin}. 
In the random graph setting, it has 
recently been used 
to understand asymptotic behavior of ERGMs, 
including singular behavior, 
{\em e.g.}, ill-posedness of 
parameter fitting 
problems~\cite{ChatterjeeDiaconis}. 

The study of the free energy in ERGMs 
dates to Park and Newman~\cite{ParkI,ParkII}, 
who used mean field and other 
non-rigorous approximations. 
For early history, see also the references in H\"{a}ggstr\"{o}m
and Jonasson~\cite{Haggstrom}. 
The first rigorous study appeared in 
Chatterjee and Diaconis~\cite{ChatterjeeDiaconis}, 
who used a large deviation approach~\cite{ChatterjeeVaradhan}. 
Radin and Yin~\cite{Radin} used the work 
of Chatterjee and Diaconis to 
formalize the notion of phases 
for ERGMs, explicitly computing phase diagrams 
for a family of two-parameter models. 
A similar three-parameter 
family was studied by Yin~\cite{Yin}.

We consider a family of directed 
exponential random graphs parametrized by 
edges and outward directed $p$-stars. 
Versions 
of this model in which 
edges and outward directed $p$-stars are 
held fixed, instead of 
controlled by a parameter, 
were studied in~\cite{AristoffZhu}. 
Such models are standard
and important in the literature of social 
networks, see e.g. Holland~\cite{Holland}, Mele~\cite{Mele} and 
the references therein. 
Many complex networks have directed structures.
Examples include email networks and social networks.
In email networks, just because user Bob's email address appears in user Alice's address book 
does not necessarily mean that the reverse is also true, although it often is, see e.g. Newman et al. \cite{NewmanII}.
In social networks, when two users Alice and Bob interact
as peers, one expects that messages will be exchanged between them in both directions.
However, if user Alice sends messages to user Bob, who is a celebrity or news source, it is likely
that user Bob will not send messages in return, see e.g. Cheng et al. \cite{Cheng}.
In this paper, we consider the statistics of edges and outward directed $p$-stars
for two main reasons. First, this is the model that is more analytically tractable than
the more general directed ERGMs, and the main results we will obtain in this paper
rely on the special properties of this model.
Second, our directed ERGM with edge and outward directed $p$-stars
has a nice microeconomic interpretation, which can be viewed as the long-run equilibrium
of a game in which the players are maximizing their own utilities. 
We will describe this in more details in the Section~\ref{NetworkSec}.

For directed 
graphs, the results of Chatterjee 
and Diaconis~\cite{ChatterjeeDiaconis} 
and Radin and Yin~\cite{Radin} do 
not directly apply, as the large 
deviation techniques they used have been developed 
only for undirected graphs.
Instead of adapting these techniques 
to the directed case, we use more
direct methods 
which lead to more precise asymptotics. 
In particular, we are able to completely 
characterize the phase behavior of 
our models. The phase diagram we find, 
with a first order phase
transition and a critical point, is  
nearly identical to what has been 
found in the undirected case~\cite{Radin}. 
But because our 
asymptotics are more precise, we can go beyond results of the 
type in~\cite{Radin} 
by studying the 
phase transition curve itself. 

The limitation of our model is that it studies the edge and outward directed $p$-stars, 
rather than the more general statistics in the more general directed ERGMs. Our method
and results cannot be applied to the undirected ERGM either. It would be certainly
be very interesting mathematically to pursue the similar results as obtained in this paper
for the more general directed ERGMs or undirected ERGMs, although it might be a big challenge
and an open problem in terms of mathematics. We study our particular model because of its
analytical tractability and exact solvability. At least for this special model, we can study
and understand the precise asymptotics rigorously and completely.

\subsection{The model} 
We consider the following exponential random graph 
model. Fix $p \ge 2$. For a 
directed graph $X$ on $n$ 
nodes and real parameters $\beta_1,\beta_2$, define
\begin{equation}\label{ERGM1}
{\mathbb P}_{n,\beta_1,\beta_2}(X) \sim 
\exp\left(\beta_1 E(X) + \frac{p!\beta_2}{n^{p-1}}S(X)\right), 
\end{equation}
where $E(X):=\sum_{i,j}X_{ij}$ 
and $S(X):=\frac{1}{p!}\sum_{i,j_{1},\ldots,j_{p}}X_{ij_{1}}X_{ij_{2}}\cdots X_{ij_{p}}$ are, 
respectively, 
the number of directed edges and 
outward directed $p$-stars in $X$, 
and $\sim$ denotes equality 
up to a normalization constant. 
We consider only directed graphs 
without duplicate edges, though loops 
will be allowed.

The following reformulation 
of~\eqref{ERGM1} will be useful. A 
simple directed graph $X$ on $n$ nodes 
is given by its adjacency matrix 
$X = (X_{ij})_{1\le i,j \le n}$ with each $X_{ij} \in \{0,1\}$. 
Here, $X_{ij} = 1$ means there is a directed edge from node $i$ to 
node $j$; otherwise, $X_{ij}=0$. 
Note that we allow $X_{ii} = 0$ or $1$, 
corresponding to the absence or 
presence, respectively, 
of a loop at node $i$. Define
\begin{align}\begin{split}\label{es}
&e(X) = n^{-2}\sum_{1\le i,j\le n} X_{ij},\\
&s(X) = n^{-p-1} \sum_{1\leq i,j_{1},j_{2},\ldots,j_{p}\leq n}X_{ij_{1}}X_{ij_{2}}\cdots X_{ij_{p}}.
\end{split}
\end{align}
Observe that 
\begin{equation*}
e(X) = n^{-2}E(X), \qquad s(X) = p!n^{-p-1}S(X).
\end{equation*}
(Note that 
we have allowed loops to contribute 
to $p$-stars; this is a minor 
point because the number of loops 
is of lower order than the number 
of directed edges.)
We think of $e(X)$ and $s(X)$ 
as homomorphism densities. 
That is, $e(X)$ is
the probability 
that a random function from a
directed edge into $X$ is a homomorphism, 
{{\em i.e.}}, an edge preserving map between the vertex sets. 
Similarly, $s(X)$ is
the probability 
that a random function from an
outward directed $p$-star into $X$ is a homomorphism. See~\cite{ChatterjeeDiaconis} for more 
details. With 
this notation, we rewrite~\eqref{ERGM1} 
as
\begin{equation}\label{ERGM}
{\mathbb P}_{n,\beta_1,\beta_2}(X) = Z_n(\beta_{1},\beta_{2})^{-1}\exp\left[n^2\left(\beta_{1} e(X) + \beta_{2} s(X)\right)\right],
\end{equation}
with 
$Z_n(\beta_1,\beta_2)$ the 
normalization constant. We 
will study  
\begin{equation*}
\psi_n(\beta_1,\beta_2) = n^{-2}\log Z_n(\beta_1,\beta_2)
\end{equation*}
as well as its limit
\begin{equation*}
\psi(\beta_1,\beta_2) := 
\lim_{n\to \infty} \psi_n(\beta_1,\beta_2).
\end{equation*}
We refer to $\psi(\beta_1,\beta_2)$ 
as the {\em free energy}. It is 
important for understanding 
the structure and statistical 
properties of the model. In 
particular, first order partial derivatives 
of $\psi$ with respect 
to $\beta_1$ and $\beta_2$ 
correspond to the limiting edge and star
densities in the model. Similarly, 
second order partial derivatives correspond 
to limiting edge and star variances. 
Consequently, singularities of $\psi(\beta_1,\beta_2)$ correspond 
to singular behavior in the model 
as the parameters vary. For instance, a singularity in a first order 
partial derivative of $\psi$ corresponds 
to a jump discontinuity of the limiting 
edge/star densities as $\beta_1,\beta_2$ 
vary across the singularity. 

We 
find below that the first 
derivative of $\psi$ is singular 
along a certain curve in the $(\beta_1,\beta_2)$ plane. This curve has an endpoint, at which the 
second derivative of $\psi$ is singular. 
A similar singularity has been 
found in the undirected version of the model; see 
Radin and Yin~\cite{Radin}. 
Our results are novel because, in 
contrast with the undirected case, we are able to obtain sharp asymptotics 
for $\psi_n$ and its partial 
derivatives at {\em finite $n$}. 
This allows us to make precise statements about the 
nature of the singularity in the model. 
In particular, we can describe the
scaling of edge and star variances 
along the singularity.
We explore 
this in detail in Sections~\ref{1.2} and~\ref{1.3}
by using an analogy with 
the grand canonical ensemble 
in statistical physics, an exponential family similar 
to~\eqref{ERGM}. See also Radin and Yin~\cite{Radin} for a similar discussion.

\subsection{Network formation}\label{NetworkSec}

In this section, let us consider a microeconomic model of network formation
that will be seen to converge in the long-run to the directed ERGM model
that we proposed. Similar network formation models in economics literature which
convergence to the equilibrium of ERGMs can be found in 
e.g. Mele \cite{Mele}, Chandrasekhar and Jackson \cite{CJ}, Badev \cite{Badev}.
 
Consider $n$ players and if there is a link from player $i$ to player $j$, 
we have $X_{ij}=1$ and it is $0$ otherwise. For simplicity, we assume $X_{ii}=0$ always. 
The link from player $i$ to player $j$ can be interpreted as an email message, or the invitation to an event
in the social networks. 
Notice that an email message may not get replied, and
an invitation may not get accepted. Thus, it fits into
the directed network setting.
For player $i$, we define his/her utility function as
\begin{equation*}
u_{i}(X)=\sum_{j=1}^{n}\beta_{1}X_{ij}+\sum_{1\leq j_{1},j_{2},\ldots,j_{p}}\frac{\beta_{2}}{n^{p-1}}X_{ij_{1}}\cdots X_{ij_{p}}.
\end{equation*}
In this setting, the player $i$ will have an incentive $\beta_{1}$ to form a link to player $j$
for any other player $j$, and will have an incentive $\frac{\beta_{2}}{n^{p-1}}$
to form links simultaneously to $p$ players $j_{1},\ldots,j_{p}$. 
Let us suppose $\beta_{1}$ is positive and $\beta_{2}$ is negative.
That means there is positive incentive to invite/message your friends, but there is negative
incentive to invite/message $p$ friends or more simultaneously. 
That can be explained the incentive to be sociable but not overly sociable.  

Over a long-period of time interval, we assume that at each (discrete) time, a player $i$ updates $X_{ij}$, $j\neq i$
to maximize his/her utility, and before the player $i$ updates the links, he/she receives
an idiosyncratic shock to his/her preferences that the econometrician cannot
observe.  The shocks are assumed to be i.i.d. logistic shocks among players and across time, 
which is a standard assumption in economics and statistics, see e.g. \cite{Train}.
Under these assumptions, the network formation process evolves according to a Markov chain 
which is irreducible and aperiodic and hence as time goes to infinity, it converges to
an equilibrium with a stationary distribution.
In the absence of random shocks, the network formation process will converge
to a Nash network as time goes to infinity, where a Nash network is a network 
in which player has no profitable deviations from his/her current linking strategy.
The random shock models unobservables that could influence the utility of additional
links, see e.g. \cite{Mele}. As time evolves, the network will converge to an equilibrium (see e.g. \cite{Mele}, \cite{CJ},
and the derivations follow the same arguments in Appendix A in \cite{Mele}) 
in which the stationary probability of observing a particular network configuration $X$ is given precisely by
\begin{equation*}
Z_n(\beta_{1},\beta_{2})^{-1}\exp\left[\beta_{1}\sum_{i,j}X_{ij} 
+ \frac{\beta_{2}}{n^{p-1}}\sum_{i,j_{1},\ldots,j_{p}}X_{ij_{1}}\cdots X_{ij_{p}}\right],
\end{equation*}
which is the directed ERGM model we defined in \eqref{ERGM}, 
where the exponent 
\begin{equation*}
\beta_{1}\sum_{i,j}X_{ij} 
+ \frac{\beta_{2}}{n^{p-1}}\sum_{i,j_{1},\ldots,j_{p}}X_{ij_{1}}\cdots X_{ij_{p}}
\end{equation*}
is known as the potential function in the economics literature, which is the combined utility
of $n$ players.

For econometricians, it is crucial to understand how to estimate the parameters $\beta_{1},\beta_{2}$
from the real world data.
Notice that the normalizing constant $Z_{n}(\beta_{1},\beta_{2})$ and hence $\psi_{n}(\beta_{1},\beta_{2})$ depend 
on the parameters $\beta_{1},\beta_{2}$ to be estimated, which brings a challenge to the MLE method.
MCMC method was proposed for the estimation, see e.g. \cite{Snijders2002}. But the MCMC method becomes
computationally expensive for large network size $n$. The alternative approach developed in recent years
is the variational inference, by directly computing and analyzing the constant $\psi_{n}(\beta_{1},\beta_{2})$
as $n\rightarrow\infty$, see e.g. \cite{MeleZhu}, which will be the focus for the rest of this paper.

\subsection{The grand canonical ensemble 
and phase transitions}\label{1.2}

To explain how our results fit into the 
phase diagram framework of~\cite{Radin}, we compare our model~\eqref{ERGM} with
the grand canonical ensemble 
from statistical physics,
which describes the statistical properties of matter in 
thermal equilibrium~\cite{Gallavotti}. 
We consider the grand canonical ensemble defined by, for $Y \subset [-n/2,n/2]^d$ with $d=2$ or $3$,
\begin{equation}\label{grandcanon}
{\mathbb P}_{n,\beta,\mu}(Y) = Z_n(\beta,\mu)^{-1}\exp\left(n^d\left[\beta \mu {\mathcal N}(Y) - \beta {\mathcal E}(Y)\right]\right),
\end{equation}
where $\mu$ is chemical 
potential, and $\beta = 1/(k_BT)$ with 
$T$ temperature and
$k_B$ Boltzmann's constant. 
Here $Z_n(\beta,\mu)$ is the 
normalization constant. Each element of $Y$ represents 
a particle, with 
${\mathcal N}(Y) = |Y|/n^d$ the density of $Y$,
and ${\mathcal E}(Y)$ 
the energy per volume of $Y$. 
A standard fact in statistical physics is that average physical properties of the model can 
be obtained from
\begin{equation*}
\psi_n(\beta,\mu):= n^{-d} \log Z_n(\beta,\mu).
\end{equation*}
In particular, the average and variance of 
${\mathcal N}(Y)$ and ${\mathcal E}(Y)$, or more generally, all 
of their moments, can be obtained 
by differentiating $\psi_n(\beta,\mu)$ with 
respect to $\beta$ or $\mu$. Usually, $n$ is very large and it is appropriate to consider
\begin{equation*}
\psi(\beta,\mu) := \lim_{n\to \infty} \psi_n(\beta,\mu),
\end{equation*}
which exists under
appropriate conditions on ${\mathcal E}$. 
One utility of this limit is that
\begin{equation}\label{switch}
\lim_{n\rightarrow\infty}\frac{\partial^{i+j}}{\partial\beta^{i}\partial\mu^{j}}
\psi_{n}(\beta,\mu)
=\frac{\partial^{i+j}}{\partial\beta^{i}\partial\mu^{j}}
\lim_{n\rightarrow\infty}\psi_{n}(\beta,\mu)
=\frac{\partial^{i+j}}{\partial\beta^{i}\partial\mu^{j}}
\psi(\beta,\mu)
\end{equation}
whenever $i,j$ are such that the derivative on the right hand side exists~\cite{Yang}. This means 
in the limit $n \to \infty$, moments of ${\mathcal N}(Y)$ and ${\mathcal E}(Y)$ 
can be computed directly from 
$\psi(\beta,\mu)$, {provided 
the appropriate partial derivatives 
of} $\psi(\beta,\mu)$ exist.

The limit $\psi(\beta,\mu)$ is key to understanding phases of matter. In particular, when ${\mathcal E}$ is 
suitably chosen, 
$\psi(\beta,\mu)$ 
is analytic except along two curves with an endpoint.
These curves correspond to the physical
transitions between solid, liquid 
and vapor phases, colloquially 
known as freezing, melting, boiling and sublimating.
The endpoint of these curves is called the {critical point}~\cite{FisherRadin}. 
See Figure 1(i). These transitions 
are {\em first order}, meaning 
the first derivative of the free 
energy has a jump discontinuity 
across the transition curve. 
On the phase transition 
curve there are  
coexisting phases of high and low 
density, and thus a nonvanishing 
variance of ${\mathcal N}(Y)$ and ${\mathcal E}(Y)$
in the limit $n \to \infty$. 

Unfortunately, rigorous analysis 
of $\psi(\beta,\mu)$ is difficult. Though the statements in the previous paragraph
are widely believed and supported by numerical experiments,  
proofs are possible only in very special 
cases~\cite{Lebowitz}. 
On the other hand, analysis of 
the ERGM free energy $\psi(\beta_1,\beta_2)$ 
is relatively tractable. Indeed, we show that 
$\psi(\beta_1,\beta_2)$ exhibits behavior 
very similar to what is conjectured 
for the grand canonical free energy $\psi(\beta,\mu)$.

\begin{figure}\label{fig1}
\begin{center}
\includegraphics[scale=0.7]{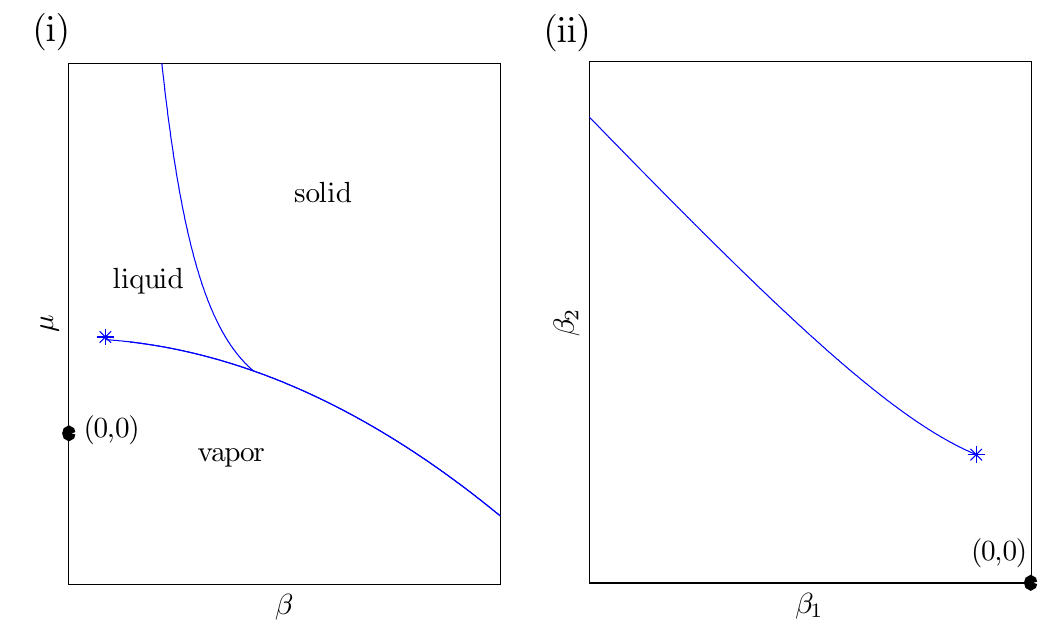}
\end{center}
\caption{Simple phase diagrams in (i) the grand canonical ensemble, 
and (ii) the ERGM model. The critical point is labeled 
with a $*$.}
\end{figure}

\subsection{Singularities of the ERGM free energy}\label{1.3}
We show that $\psi(\beta_1,\beta_2)$ 
is analytic except along a certain 
curve, which we call the {\em 
phase transition curve}. The 
curve has an endpoint, which 
we call the {\em critical point}. 
We prove that  
on the phase transition curve but away from the critical point, 
the first order partial derivatives of $\psi(\beta_1,\beta_2)$ have a jump 
discontinuity. 
Moreover, at the critical point, the first order partial 
derivatives of $\psi(\beta_1,\beta_2)$ are continuous but the second 
order derivatives diverge. 
Precisely the same behavior is 
believed to occur on the 
liquid-vapor transition curve in the  
grand canonical ensemble.
See Figure~1.

To understand these singularities better, consider the 
following. Just as in~\eqref{switch},
\begin{align}\begin{split}\label{switch2}
&\lim_{n\to \infty}\frac{\partial}{\partial \beta_i} \psi_n(\beta_1,\beta_2)
= \frac{\partial}{\partial \beta_i}  \psi(\beta_1,\beta_2), \quad 
i \in \{1,2\}\\
&\lim_{n\to \infty}\frac{\partial^2}{\partial \beta_i \partial \beta_j} \psi_n(\beta_1,\beta_2)
= \frac{\partial^2}{\partial \beta_i\partial \beta_j}  \psi(\beta_1,\beta_2), \quad i,j\in\{1,2\},\end{split}
\end{align}
if the derivatives on the right hand side exist. 
The commuting of limits follows from very 
general arguments of Yang and Lee~\cite{Yang}. 
Though their proof is written in the 
statistical mechanics setting of 
Section~\ref{1.2}, their 
arguments go  
through without any difficulties in 
our case. See also Radin and Yin~\cite{Radin} for remarks on this issue in the undirected graph setting.

Next, from simple computations,
\begin{align}\begin{split}\label{derivs2}
&\quad\frac{\partial}{\partial \beta_1} \psi_n(\beta_1,\beta_2) 
= {\mathbb E}_n[e(X)],\quad \frac{\partial}{\partial \beta_2} \psi_n(\beta_1,\beta_2) 
= {\mathbb E}_n[s(X)]\\
&\quad \frac{\partial^2}{\partial \beta_1^2} \psi_n(\beta_1,\beta_2) 
= n^2{\hbox{Var}_n}(e(X)), \quad \frac{\partial^2}{\partial \beta_2^2} \psi_n(\beta_1,\beta_2) = n^2{\hbox{Var}_n}(s(X)) \\
&\frac{\partial}{\partial \beta_1\partial \beta_2} \psi_n(\beta_1,\beta_2) 
= 
\frac{\partial}{\partial \beta_2\partial \beta_1} \psi_n(\beta_1,\beta_2) 
= n^2{\hbox{Cov}_n}(e(X),s(X)) .
\end{split}
\end{align}
Thus, a jump discontinuity in $\partial\psi(\beta_1,\beta_2)/\partial \beta_1$ (resp. $\partial\psi(\beta_1,\beta_2)/\partial \beta_2$)
along the transition curve implies 
a jump in the average value of $e(X)$ (resp. $s(X)$) 
across the curve in the limit $n\to \infty$.
Similarly, at the critical point, divergence 
of $\partial^2\psi(\beta_1,\beta_2)/\partial \beta_1^2$ (resp. $\partial^2\psi(\beta_1,\beta_2)/\partial \beta_2^2$)
implies that the variance of $e(X)$ (resp. $s(X)$) 
decays more slowly than 
$n^{-2}$. Away from the transition curve, all partial 
derivatives of $\psi(\beta_1,\beta_2)$ of all orders 
exist and are finite, so in particular 
the variance of $e(X)$ and $s(X)$ decays at least as fast as 
$n^{-2}$.  More detailed statements 
would require an analysis $\psi_n(\beta_1,\beta_2)$ for {\em finite} $n$; this is 
much more difficult to study than the limit 
$\psi(\beta_1,\beta_2)$. See~\cite{ChatterjeeDiaconis,Radin} for 
an analysis of the free energy in the undirected 
version of the model.

In the context of social networks (Section \ref{NetworkSec}), 
the quantities computed in \eqref{derivs2} reveal the statistics of the networks
at its long-run equilibrium. For instance, $\mathbb{E}_{n}[e(X)]$
denotes the average number of outward links from an average player. 
The existence of phase transitions tells us that as
the direct benefit $\beta_{1}$ and $\beta_{2}$ a player receives in his/her utility 
vary smoothly, many important statistics of the network, and hence the network structure 
at its long-run equilibrium may vary non-smoothly.

Our results rely on new and 
precise asymptotics for $\psi_n(\beta_1,\beta_2)$ and its partial derivatives. 
We use our asymptotics to calculate $\psi(\beta_1,\beta_2)$ and the scaling of the 
variance/covariance of 
$e(X)$ and $s(X)$ on the phase 
transition curve. The formula for $\psi(\beta_1,\beta_2)$ 
resembles the analogous 
free energy in the undirected 
version of the model. However, 
the behavior of our model 
along the phase transition 
curve is qualitatively different from 
the undirected case; see the 
discussion after Theorem~\ref{covariance} below.

The study of 
scaling on the transition curve has 
not 
been done before, in both the directed and 
undirected versions of the model. 
(This is because the scaling cannot 
be computed from $\psi(\beta_1,\beta_2)$ alone, 
as~\eqref{switch2} does not hold along the 
phase transition curve.)
We find that the variances 
of $e(X)$ and $s(X)$ vanish on the 
transition curve as $n \to \infty$, 
while 
the edge probability 
between fixed nodes is a 
Bernoulli random variable
whose parameter 
is a {convex combination} 
of the expected values of $e(X)$ just 
above and below the
curve. Combining these results, 
we show below that large graphs 
do not look like Erd\H{o}s-R\'{e}nyi 
random graphs with a binary 
distributed parameter -- 
that is, the first order phase transition 
does not correspond to phase 
coexistence in the usual sense. 
This is 
unexpected in light of the statistical 
physics analogy above. 
See the discussion 
below Theorem~\ref{marginaldensities}.

The remainder of this paper is organized as follows. 
Main results are stated in Section~\ref{THEOREMS}. 
The results are obtained by estimates, stated 
in Section~\ref{ESTIMATES}, which allow 
for a precise computation of $\psi_n(\beta_1,\beta_2)$ and derivatives thereof. 
All proofs are in Section~\ref{PROOFS}.

\section{Notation and results}\label{THEOREMS}

Our main results rely on the following trick. Observe that we can rewrite 
\begin{equation*}
e(X) = n^{-2} \sum_{i=1}^n\left(\sum_{j=1}^n X_{ij}\right), \qquad s(X) = n^{-p-1}\sum_{i=1}^n\left(\sum_{j=1}^n X_{ij}\right)^p,
\end{equation*}
where $\sum_{j=1}^n X_{ij}$, 
$i=1,\ldots,n$, are independent 
random variables. 
Thus, we can calculate $Z_n(\beta_1,\beta_2)$,
and hence also 
$\psi_n(\beta_1,\beta_2)$, as  
``one-dimensional'' objects. See 
Proposition~\ref{Zn} below for details.
Our analysis will depend heavily on the function
\begin{equation*}
\ell(x) := \beta_{1} x + \beta_{2} x^p-x \log x - (1-x)\log (1-x).
\end{equation*}
It is easy to see that $\ell$ is analytic 
in $(0,1)$ and continuous on $[0,1]$. 
Note that $\ell$ is essentially identical to the 
function of the same name studied in~\cite{Radin}: 
after multiplying $\beta_1$ and $\beta_2$ by two 
the functions differ only by a constant. This 
allows us to use results from~\cite{Radin} 
concerning $\ell$.

Our first result is the following formula for the 
free energy.
\begin{theorem}\label{free_energy}
 For any $\beta_1$, $\beta_2$, as $n \to \infty$ we have
\begin{equation*}
\psi_n(\beta_1,\beta_2) = \max_{x \in [0,1]} \ell(x)+O(n^{-1}\log n).
\end{equation*}
\end{theorem}
In particular, letting $n\to \infty$, we get 
\begin{equation*}
\psi(\beta_1,\beta_2) = \max_{x \in [0,1]} \ell(x).
\end{equation*}
Essentially the same formula holds in 
the undirected case~\cite{Radin}. 
Thus, the phase diagram for 
our model is the 
same as in the undirected 
version, after multiplying
$\beta_1$ and $\beta_2$ by $2$. 
The following result from~\cite{Radin} 
characterizes the curve along 
which $\psi(\beta_1,\beta_2)$ 
is singular.
\begin{theorem}[Radin and Yin~\cite{Radin}]\label{trans_curve}
There is a certain curve in the 
$(\beta_1,\beta_2)$-plane with the endpoint
\begin{equation*}
(\beta_{1}^{c},\beta_{2}^{c})=\left(\log(p-1) - \frac{p}{p-1},\frac{p^{p-1}}{(p-1)^p}\right),
\end{equation*}
such that off the curve and at the endpoint, $\ell$ has a 
unique global maximizer $x^* \in (0,1)$, while on the curve away from the 
endpoint, $\ell$ has two global 
maximizers, $x_{1}^{*}$ and $x_{2}^{*}$, 
with $0 < x_1^* < (p-1)/p < x_2^* < 1$.
\end{theorem}
The curve in Theorem~\ref{trans_curve} will be called the 
{\em phase transition curve} and written $\beta_2 = q(\beta_1)$. The endpoint 
will be called the {\em critical point}.  
Though our free energy $\psi(\beta_1,\beta_2)$ is essentially the same 
as its undirected counterpart, 
the behavior of our 
model on the phase transition curve 
is qualitatively different 
from the undirected version, as 
we will see below. 

It is not possible to 
write an explicit equation for the 
phase transition curve 
in general; see~\cite{Radin} 
for a graph obtained numerically. 
However, in~\cite{Radin} is shown that $q(\beta_{1})$ 
is continuous and decreasing in $\beta_{1}$, 
with $\lim_{\beta_{1}\rightarrow-\infty}|q(\beta_{1})+\beta_{1}|=0$.
We have the following more precise result.

\begin{theorem}\label{PropertyCurve}
(i) $q(\beta_{1})$ is differentiable for $\beta_{1}<\beta_{1}^{c}$ with
\begin{equation*}
q'(\beta_{1})=-\frac{x_{1}^{\ast}-x_{2}^{\ast}}{(x_{1}^{\ast})^{p}-(x_{2}^{\ast})^{p}}<0.
\end{equation*}
In particular,
\begin{equation*}
\lim_{\beta_{1}\rightarrow\beta_{1}^{c}}q'(\beta_{1})=-\frac{p^{p-2}}{(p-1)^{p-1}},
\qquad
\text{and}
\qquad
\lim_{\beta_{1}\rightarrow-\infty}q'(\beta_{1})=-1.
\end{equation*}

(ii) $q(\beta_{1})$ is convex in $\beta_{1}$.
\end{theorem}
When $p = 2$, along the line $\beta_1 + \beta_2 = 0$ the function $\ell$ is 
symmetric around $1/2$. It follows that $x_1^* + x_2^* = 1$ along this 
line, so Theorem~\ref{PropertyCurve} implies $\ell(\beta_1) = -\beta_1$. 
See Figure~2(i).

The following theorems give the scaling of 
the variance and 
covariance of $e(X)$ and $s(X)$. 
See~\cite{Radin} for 
computations of these
quantities 
off the phase transition curve 
in the undirected graph case.
Here we compute the scaling of 
the variance and covariance at all
$(\beta_1,\beta_2)$, including on 
the phase transition curve and at the critical point. On 
the transition curve, we  
use 
 precise asymptotics 
for partial derivatives 
of $\psi_n(\beta_1,\beta_2)$ 
to obtain the scalings. 
We emphasize that
such asymptotics 
cannot be obtained directly from
large deviations 
techniques of the type used in~\cite{ChatterjeeDiaconis,Radin}.

\begin{theorem}\label{MainThm}
Off the phase transition curve, 
\begin{equation*}
\lim_{n\to \infty}\frac{\partial^2 }{\partial \beta_1^2}\psi_n(\beta_1,\beta_2) 
= \frac{\partial^2 }{\partial \beta_1^2}\lim_{n\to \infty}\psi_n(\beta_1,\beta_2)
= \frac{1}{|\ell''(x^*)|}.
\end{equation*}
On the phase transition curve except at the critical point, 
\begin{equation*}
\lim_{n\rightarrow\infty}\frac{1}{n}\frac{\partial^2 }{\partial \beta_1^2}\psi_n(\beta_1,\beta_2) 
=\frac{(x_{1}^{\ast}-x_{2}^{\ast})^{2}\sqrt{x_{1}^{\ast}(1-x_{1}^{\ast})|\ell''(x_{1}^{\ast})|}
\sqrt{x_{2}^{\ast}(1-x_{2}^{\ast})|\ell''(x_{2}^{\ast})|}}
{\left(\sqrt{x_{1}^{\ast}(1-x_{1}^{\ast})|\ell''(x_{1}^{\ast})|}
+\sqrt{x_{2}^{\ast}(1-x_{2}^{\ast})|\ell''(x_{2}^{\ast})|}\right)^{2}}.
\end{equation*}
At the critical point, 
\begin{equation*}
\lim_{n\rightarrow\infty}\frac{1}{n^{1/2}}
\frac{\partial^2 }{\partial \beta_1^2}\psi_n(\beta_1,\beta_2)
=\frac{\Gamma(\frac{3}{4})}{\Gamma(\frac{1}{4})}
\frac{2\sqrt{6}(p-1)}{p^{5/2}}. 
\end{equation*}
\end{theorem}

\begin{theorem}\label{starvariance}
Off the phase transition curve, 
\begin{equation*}
\lim_{n\to \infty}\frac{\partial^2 }{\partial \beta_2^2}\psi_n(\beta_1,\beta_2) 
= \frac{\partial^2 }{\partial \beta_2^2}\lim_{n\to \infty}\psi_n(\beta_1,\beta_2)
= \frac{p^{2}(x^{\ast})^{2p-2}}{|\ell''(x^*)|}.
\end{equation*}
On the transition curve except at the critical point, 
\begin{equation*}
\lim_{n\rightarrow\infty}\frac{1}{n}\frac{\partial^2 }{\partial \beta_2^2}\psi_n(\beta_1,\beta_2)
=\frac{((x_{1}^{\ast})^{p}-(x_{2}^{\ast})^{p})^{2}\sqrt{x_{1}^{\ast}(1-x_{1}^{\ast})|\ell''(x_{1}^{\ast})|}
\sqrt{x_{2}^{\ast}(1-x_{2}^{\ast})|\ell''(x_{2}^{\ast})|}}
{\left(\sqrt{x_{1}^{\ast}(1-x_{1}^{\ast})|\ell''(x_{1}^{\ast})|}
+\sqrt{x_{2}^{\ast}(1-x_{2}^{\ast})|\ell''(x_{2}^{\ast})|}\right)^{2}}.
\end{equation*}
At the critical point, 
\begin{equation*}
\lim_{n\rightarrow\infty}\frac{1}{n^{1/2}}\frac{\partial^2 }{\partial \beta_2^2}\psi_n(\beta_1,\beta_2) 
=\frac{2\sqrt{6}\Gamma(\frac{3}{4})}{\Gamma(\frac{1}{4})}\frac{(p-1)^{2p-1}}{p^{2p-\frac{3}{2}}}.
\end{equation*}
\end{theorem}

\begin{theorem}\label{covariance}
Off the phase transition curve, 
\begin{equation*}
\lim_{n\to \infty}\frac{\partial^{2}}{\partial\beta_{1}\partial\beta_{2}}\psi_n(\beta_1,\beta_2) 
=\frac{\partial^{2}}{\partial\beta_{1}\partial\beta_{2}}\lim_{n\to \infty}\psi_n(\beta_1,\beta_2)
=\frac{p(x^{\ast})^{p-1}}{|\ell''(x^*)|}.
\end{equation*}
On the transition curve except at the critical point, 
\begin{align*}
&\lim_{n\rightarrow\infty}\frac{1}{n}
\frac{\partial^2 }{\partial \beta_1 \partial \beta_2}\psi_n(\beta_1,\beta_2)
\\
&=
\frac{((x_{1}^{\ast})^{p}-(x_{2}^{\ast})^{p})(x_{1}^{\ast}-x_{2}^{\ast})
\sqrt{x_{1}^{\ast}(1-x_{1}^{\ast})|\ell''(x_{1}^{\ast})|}
\sqrt{x_{2}^{\ast}(1-x_{2}^{\ast})|\ell''(x_{2}^{\ast})|}}
{\left(\sqrt{x_{1}^{\ast}(1-x_{1}^{\ast})|\ell''(x_{1}^{\ast})|}
+\sqrt{x_{2}^{\ast}(1-x_{2}^{\ast})|\ell''(x_{2}^{\ast})|}\right)^{2}}.
\end{align*}
At the critical point, 
\begin{equation*}
\lim_{n\rightarrow\infty}\frac{1}{n^{1/2}}
\frac{\partial^2 }{\partial \beta_1 \partial \beta_2}\psi_n(\beta_1,\beta_2)
=\frac{2\sqrt{6}\Gamma(\frac{3}{4})}{\Gamma(\frac{1}{4})}\frac{(p-1)^{p}}{p^{p+\frac{1}{2}}}.
\end{equation*}
\end{theorem}

See Figure~2 for a comparison 
of these results with 
Monte Carlo simulation.
By Theorems~\ref{MainThm}--\ref{starvariance} and~\eqref{derivs2}, the variances 
of $e(X)$ and $s(X)$
are order $n^{-2}$ off
 the transition curve, 
 order $n^{-1}$ on 
the phase transition curve 
away from the critical 
point, and order
$n^{-3/2}$ at the 
critical point. In 
particular, the variances 
of $e(X)$ and $s(X)$ 
vanish on the transition 
curve as $n \to \infty$.

\begin{theorem}\label{marginaldensities}
Off the phase transition curve and at the critical point,
\begin{equation}
\lim_{n\rightarrow\infty}\mathbb{P}_{n}(X_{12}=1)
=x^{\ast}.
\end{equation}
On the phase transition curve except at the critical point,
\begin{equation}
\lim_{n\rightarrow\infty}\mathbb{P}_{n}(X_{12}=1)
=\alpha x_{1}^{\ast}+(1-\alpha)x_{2}^{\ast},
\end{equation}
where
\begin{equation}\label{alphaEqn}
\alpha:=\frac{\sqrt{
x_{2}^{\ast}(1-x_{2}^{\ast})|\ell''(x_{2}^{\ast})|}}
{\sqrt{x_{1}^{\ast}(1-x_{1}^{\ast})|\ell''(x_{1}^{\ast})|}
+\sqrt{x_{2}^{\ast}(1-x_{2}^{\ast})|\ell''(x_{2}^{\ast})|}}.
\end{equation}
\end{theorem}

To put our results in perspective, we compare with 
the undirected case. Since 
the behavior of the directed 
and undirected models off 
the phase transition curve is 
similar, our discussion focuses on the phase 
transition curve.

In the undirected case 
(see Theorem 3.4 of~\cite{Radin}), on the phase transition curve away from the critical point, for large $n$ a typical graph looks like a sample from $G(n,p^{\ast})$, 
where $p^{\ast}$ is a distribution on the two global maximizers $x^{\ast}_{1}<x^{\ast}_{2}$ of 
$\ell$. This distribution 
is not specified in~\cite{Radin}. 
However, it is easy to check 
that the proof of Theorem~\ref{marginaldensities} goes 
through
for undirected graphs.  
In particular, since 
$\alpha \in (0,1)$ and 
$x_1^\ast < x_2^\ast$, the distribution 
of $p^{\ast}$ is nontrivial 
(i.e., not deterministic), which 
indicates phase coexistence.
Since for large $n$ a graph 
looks like a sample from either $G(n,x_1^\ast)$ 
or $G(n,x_2^\ast)$, both 
with positive probability,  
the variances of $e(X)$ and 
$s(X)$ do not vanish. 
This is expected 
along first order phase 
transitions, as discussed 
in Section~\ref{1.3} above.

The situation in our directed 
graph model is qualitatively 
different. Let $\vec{G}(n,p)$ be 
the directed Erd\H{o}s-R\'{e}nyi 
graph on $n$ nodes, in which there 
is a directed edge 
between each ordered pair of nodes 
with probability $p$. 
Along the phase transition 
curve, since the variances of $e(X)$ 
and $s(X)$ vanish as $n \to \infty$, a typical large graph 
in our model does {\em not} behave like 
$\vec{G}(n,p^\ast)$ with $p^\ast$ 
sampled from 
a nontrivial distribution on 
$x_1^*$ and $x_2^*$. Thus, 
there is no phase coexistence 
in the sense described above. 
We do not have a more 
precise result 
about graph structure 
along the transition curve, 
but we suspect the 
following is true. 
For large $n$, on the phase 
transition curve away 
from the critical point, 
a typical graph is 
``bipodal'': there 
is a node set of 
size
$\approx \alpha n$ 
in which each node has 
an outward edge to any 
other node with probability
$\approx x_1^*$, and another 
node set of size 
$\approx (1-\alpha)n$ 
in which each node has 
an outward edge to any 
other node with
probability 
$\approx x_2^*$. 
Indeed, a similar result has 
been found in a closely 
related model~\cite{AristoffZhu}.
See also~\cite{Aristoff}.

Finally, we obtain the asymptotics of the joint 
distribution of two directed edges, e.g. $X_{12}$ and $X_{34}$
or $X_{12}$ and $X_{13}$. 

\begin{theorem}\label{marginaldensities2}
(i) Off the phase transition curve and at the critical point,
\begin{equation*}
\lim_{n\rightarrow\infty}\mathbb{P}_{n}(X_{12}=1,X_{34}=1)
=(x^{\ast})^{2}.
\end{equation*}
On the phase transition curve except at the critical point,
\begin{equation*}
\lim_{n\rightarrow\infty}\mathbb{P}_{n}(X_{12}=1,X_{34}=1)
=(\alpha x_{1}^{\ast}+(1-\alpha)x_{2}^{\ast})^{2},
\end{equation*}
where $\alpha$ is defined in \eqref{alphaEqn}.

(ii) Off the phase transition curve and at the critical point,
\begin{equation*}
\lim_{n\rightarrow\infty}\mathbb{P}_{n}(X_{12}=1,X_{13}=1)
=(x^{\ast})^{2}.
\end{equation*}
On the phase transition curve except at the critical point,
\begin{equation*}
\lim_{n\rightarrow\infty}\mathbb{P}_{n}(X_{12}=1,X_{13}=1)
=\alpha(x_{1}^{\ast})^{2}+(1-\alpha)(x_{2}^{\ast})^{2},
\end{equation*}
where $\alpha$ is defined in \eqref{alphaEqn}.
\end{theorem}

\begin{remark}
(i) The result in Theorem~\ref{marginaldensities2} (i) holds
for any $\lim_{n\rightarrow\infty}\mathbb{P}_{n}(X_{1i}=1,X_{jk}=1)$
as long as $j\neq 1$ and $i\neq 1$, $k\neq j$.

(ii) The formulas for $\lim_{n\rightarrow\infty}\mathbb{P}_{n}(X_{12}=1,X_{34}=0)$
and $\lim_{n\rightarrow\infty}\mathbb{P}_{n}(X_{12}=1,X_{13}=0)$ etc. follow
directly from Theorem~\ref{marginaldensities} and Theorem~\ref{marginaldensities2}.

(iii) The results in Theorem~\ref{marginaldensities2} show that
two distinct edges are asymptotically independent off the phase transition
curve and at the critical point, and if two directed edges do not share the same root,
then they are asymptotically independent even at the critical point.
\end{remark}

\begin{figure}
\begin{center}
\includegraphics[scale=0.7]{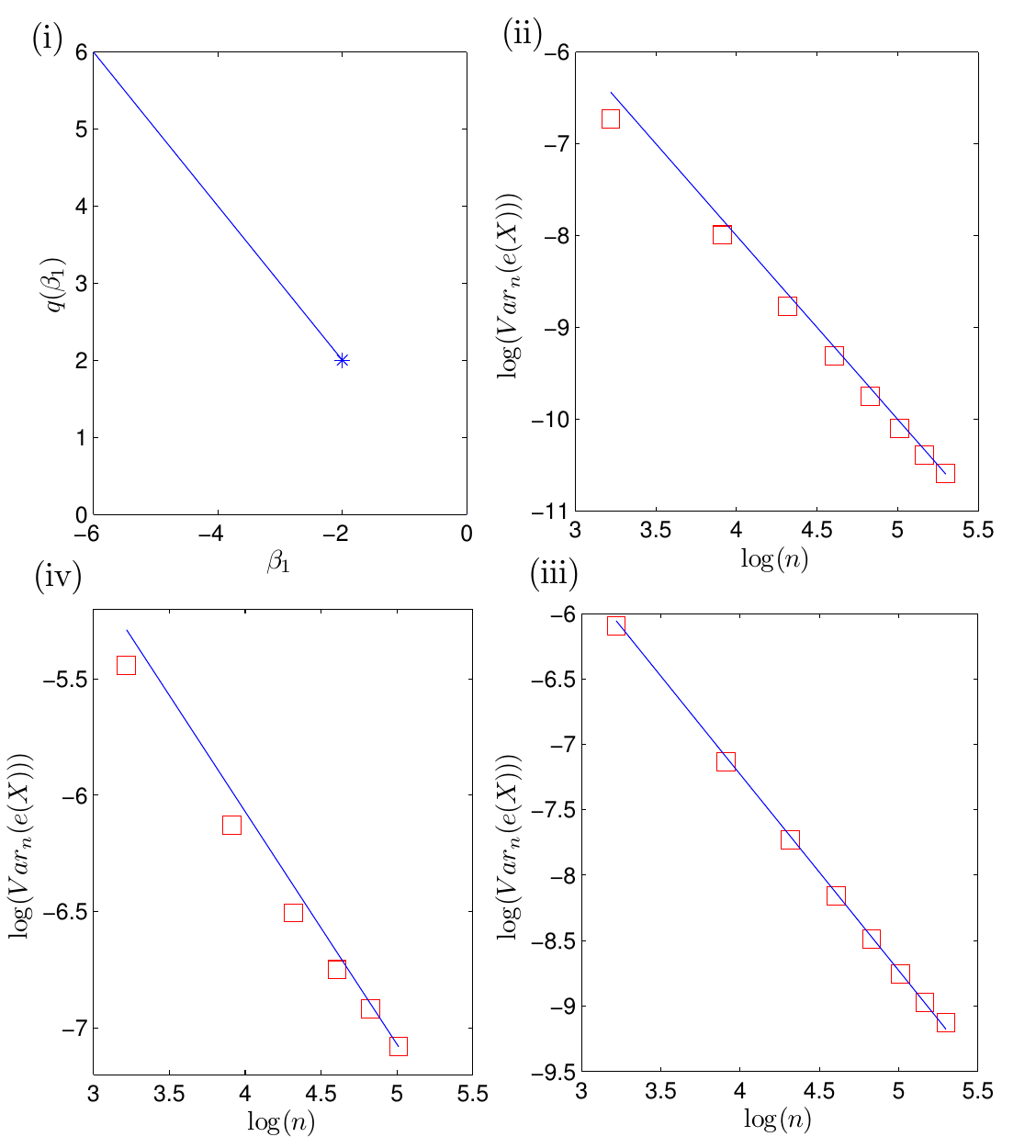}
\end{center}
\caption{(i): The graph of the phase transition curve $\beta_2 = q(\beta_1)$ 
when $p = 2$, with the critical point labeled by $*$. (ii)-(iv): Scaling of the variance of $e(X)$ (ii) off the 
phase transition curve, (iii) at the critical point, and (iv) on the phase transition 
curve away from the critical point. For (ii)-(iv) we use $p = 2$ and 
$(\beta_1,\beta_2)$ values of 
$(-3/2,3/2)$, $(-2,2)$ and $(-5/2,5/2)$, respectively. The straight lines are 
obtained from the scaling in Theorem~\ref{MainThm}, and the 
squares are obtained by Monte Carlo simulation.}
\label{fig2}
\end{figure}

\section{Key Estimates}\label{ESTIMATES}

First we have the following formula for the 
normalization $Z_n(\beta_1,\beta_2)$:
\begin{proposition}\label{Zn}
Let $W$ be a binomial random variable 
with parameters $n$ and $\frac{1}{2}$: 
\begin{equation*}
\mathbb{P}(W = i)= 2^{-n}\binom{n}{i}. 
\end{equation*}
Then 
\begin{equation*}
 Z_{n}(\beta_{1},\beta_{2}) = 
2^{n^2} \left({\mathbb E}\left[\exp\left(\beta_{1} W + \frac{\beta_{2}}{n^{p-1}}W^p\right)\right]\right)^n.
\end{equation*}
\end{proposition}

Next we approximate the expectation in Proposition~\ref{Zn} in terms of an integral:
\begin{proposition}\label{E}
Let $W$ be a binomial random variable 
with parameters $n$ and $\frac{1}{2}$. 
Then for any $r < 1$, 
\begin{align*}
&{\mathbb E}\left[W^k\exp\left(\beta_{1} W + \frac{\beta_{2}}{n^{p-1}}W^p\right)\right] 
\\
&= \begin{cases}  \left(1+O\left(n^{1/2-r}\right)\right)
n^k 2^{-n}\sqrt{\frac{n}{2\pi}}\int_{0}^{1} 
\sqrt{\frac{x^{2k}}{x(1-x)}}e^{n\ell(x)}\,dx, & (\beta_1,\beta_2) \ne 
(\beta_1^c, \beta_2^c)\\
  \left(1+O\left(n^{1/4-r}\right)\right)
n^k 2^{-n}\sqrt{\frac{n}{2\pi}}\int_{0}^{1} 
\sqrt{\frac{x^{2k}}{x(1-x)}}e^{n\ell(x)}\,dx, &(\beta_1,\beta_2) = 
(\beta_1^c, \beta_2^c) \end{cases}
\end{align*}
\end{proposition}

Lastly we give a technical lemma
for computing the integral in Proposition~\ref{E}:
\begin{proposition}\label{laplace}
Let $f$ be an analytic function in $(0,1)$ with 
Taylor expansion at $c \in (0,1)$ given by 
\begin{equation*}
 f(x) = d_0(c) + d_1(c)(x-c) + d_2(c)(x-c)^2 + \ldots, \quad d_j(c) := \frac{f^{(j)}(c)}{j!}.
\end{equation*}
For $c \in (0,1)$, define
\begin{align}\begin{split}\label{bda}
 &b_k(c) =  \frac{\ell^{k}(c)}{k!},\\
 &{\alpha}_k(c) =  \Gamma\left(\frac{k}{2}\right)|b_2(c)|^{-k/2},\\
 &{\gamma}_k(c) =  \frac{1}{2}\Gamma\left(\frac{k}{4}\right)|b_4(c)|^{-k/4}.\end{split}
\end{align}
Assume that $f(x)e^{n\ell(x)} \in L^1[0,1]$ for each $n$. 
Then as $n\rightarrow\infty$, we have the following.
\begin{itemize}
\item[(i)] Off the phase transition curve, 
\begin{equation*}
\int_0^1 f(x) e^{n\ell(x)}\,dx = e^{n\ell(c)}\left[n^{-1/2}d_0\alpha_1 +
 n^{-3/2}\Lambda + O(n^{-5/2})\right]
\end{equation*}
where 
\begin{equation*}
\Lambda := d_2\alpha_3 + d_1 b_3\alpha_5 + d_0 b_4 \alpha_5 + \frac{1}{2}d_0 b_3^2 \alpha_7
\end{equation*}
with $c = x^*$ the unique maximizer 
of $\ell$, and $d_j = d_j(c)$, $b_j = b_j(c)$, $\alpha_j = \alpha_j(c)$. 
\item[(ii)] On the phase transition curve except 
at the critical point, 
\begin{equation*}
\int_0^1 f(x) e^{n\ell(x)}\,dx = 
e^{n\ell(c)}\left[n^{-1/2}\left(d_0(c_1)\alpha_1(c_1) +
 d_0(c_2)\alpha_1(c_2)\right) + O(n^{-3/2})\right]
\end{equation*}
where $c_1$ and $c_2$ are the maximizers of $\ell$.
\item[(iii)] At the critical point, 
\begin{equation*}
\int_0^1 f(x) e^{n\ell(x)}\,dx = e^{n\ell(c)}\left[n^{-1/4}d_0\gamma_1 +
n^{-3/4}\Theta  + O(n^{-5/4})\right]
\end{equation*}
where 
\begin{equation*}
\Theta := d_2\gamma_3 + d_1 b_5\gamma_7 + 
d_0 b_6 \gamma_7 + \frac{1}{2}d_0 b_5^2 \gamma_{11}
\end{equation*}
with $c = x^*$ the unique maximizer 
of $\ell$, and $d_j = d_j(c)$, $b_j = b_j(c)$, $\gamma_j = \gamma_j(c)$.
\end{itemize}
\end{proposition}

Note that this strategy allows for a relatively 
precise computation 
of $Z_n(\beta_1,\beta_2)$. Unfortunately, 
arbitrary precision cannot 
be achieved, due to the error inherent in 
the sum to integral approximation of
Proposition~\ref{E}.

\section{Proofs}\label{PROOFS}

Before turning to the proofs of the theorems in Section~\ref{THEOREMS}, 
we will prove the estimates in Section~\ref{ESTIMATES}.
The following result will be needed in almost all of our proofs.
\begin{proposition}\label{order}
Off the phase transition curve,
\begin{equation*}
\ell'(x^*) = 0,\quad \ell''(x^*) < 0.
\end{equation*}
On the phase transition curve except at the critical point, 
\begin{equation*}
\ell'(x_1^*) = \ell'(x_2^*) = 0,\quad \ell''(x_1^*)<0, \quad \ell''(x_2^*) < 0.
\end{equation*}
At the critical point,
\begin{equation*}
\ell'(x^*)=\ell''(x^*) 
= \ell'''(x^*) = 0, \quad \ell^{(4)}(x^*)=\frac{-p^{5}}{(p-1)^{2}}< 0.
\end{equation*}
\end{proposition}

\begin{proof}
It is straightforward to compute that
\begin{align*}
&\ell'(x)=\beta_{1}+p\beta_{2}x^{p-1}-\log\left(\frac{x}{1-x}\right),
\\
&\ell''(x)=p(p-1)\beta_{2}x^{p-2}-\frac{1}{x}-\frac{1}{1-x},
\\
&\ell'''(x)=p(p-1)(p-2)\beta_{2}x^{p-3}
+\frac{1}{x^{2}}-\frac{1}{(1-x)^{2}},
\\
&\ell^{(4)}(x)=p(p-1)(p-2)(p-3)\beta_{2}x^{p-4}
-\frac{2}{x^{3}}-\frac{2}{(1-x)^{3}}.
\end{align*}
Since $\lim_{x\rightarrow 0^{+}}\ell'(x)=+\infty$
and $\lim_{x\rightarrow 1^{-}}\ell'(x)=-\infty$, the maximum
is achieved at a local maximum, we have $\ell'(x^{\ast})=0$.

Let us first show that $\ell''(x^{\ast})<0$ off the critical point 
(where $x^*$ denotes either $x_1^*$ or $x_2^*$ if we are on the 
phase transition curve).
Following the proof of Proposition 3.2 in Radin and Yin~\cite{Radin},
we first analyze the properties of $\ell''(x)$. 
We can re-write $\ell''(x)$ as
\begin{equation}
\ell''(x)=x^{p-2}p(p-1)\left[\beta_{2}-\frac{1}{p(p-1)x^{p-1}(1-x)}\right].
\end{equation}
Consider the function
\begin{equation}
m(x):=\frac{1}{p(p-1)x^{p-1}(1-x)}.
\end{equation}
It is easy to observe that $m(x)\geq\frac{p^{p-1}}{(p-1)^{p}}$
and the equality holds if and only if $x=\frac{p-1}{p}$.

(i) If $\beta_{2}<\frac{p^{p-1}}{(p-1)^{p}}$, $\ell''(x)<0$ on $[0,1]$
and in particular $\ell''(x^{\ast})<0$.

(ii) If $\beta_{2}>\frac{p^{p-1}}{(p-1)^{p}}$, there exist $0<x_{1}<\frac{p-1}{p}<x_{2}<1$
so that $\ell''(x)<0$ on $0<x<x_{1}$, $\ell''(x)>0$ on $x_{1}<x<x_{2}$
and $\ell''(x)<0$ on $x_{2}<x<1$. Moreover $\ell''(x_{1})=\ell''(x_{2})=0$.
If $\ell'(x_{1})\geq 0$, $\ell(x)$ has a unique local and hence global maximizer $x^{\ast}>x_{2}$;
if $\ell'(x_{2})\leq 0$,  $\ell(x)$ has a unique local and hence global maximizer 
$x^{\ast}<x_{1}$. Finally, if $\ell'(x_{1})<0<\ell'(x_{2})$, then $\ell(x)$
has two local maximizers $x_{1}^{\ast}$ and $x_{2}^{\ast}$ 
so that $x_{1}^{\ast}<x_{1}<\frac{p-1}{p}<x_{2}<x_{2}^{\ast}$. Since $\ell''$ vanishes
only at $x_{1}$ and $x_{2}$, we have proved that $\ell''(x^{\ast})<0$.

(iii) If $\beta_{2}=\frac{p^{p-1}}{(p-1)^{p}}$, $\ell''(x)\leq 0$ on $[0,1]$
and $\ell''(x)=0$ if and only if $x=\frac{p-1}{p}$ by the properties of $m(x)$.
Therefore, $\ell''(x^{\ast})=0$ if and only if $x^{\ast}=\frac{p-1}{p}$.
Since $\ell'(x^{\ast})=0$, $x^{\ast}=\frac{p-1}{p}$ if and only if
\begin{equation}
\beta_{1}=-p\frac{p^{p-1}}{(p-1)^{p}}\left(\frac{p-1}{p}\right)^{p-1}
+\log\left(\frac{\frac{p-1}{p}}{1-\frac{p-1}{p}}\right)=\beta_{1}^{c},
\end{equation}
Hence $\ell''(x^{\ast})<0$ off the critical point
and $\ell''(x^{\ast})=0$ at the critical point.

Furthermore, at the critical point $(\beta_{1},\beta_{2})=(\beta_{1}^{c},\beta_{2}^{c})$,
we can compute that 
\begin{equation*}
\ell'''(x^{\ast})=p(p-1)(p-2)\frac{p^{p-1}}{(p-1)^{p}}\frac{(p-1)^{p-3}}{p^{p-3}}
+\frac{p^{2}}{(p-1)^{2}}-p^{2}=0.
\end{equation*}
Moreover,
\begin{align*}
\ell^{(4)}(x^{\ast})&=p(p-1)(p-2)(p-3)\frac{p^{p-1}}{(p-1)^{p}}\frac{(p-1)^{p-4}}{p^{p-4}}
-\frac{2p^{3}}{(p-1)^{3}}-2p^{3}
\\
&=\frac{-p^{5}}{(p-1)^{2}}<0.
\end{align*}
\end{proof}

The next three proofs are for the results in 
Section~\ref{ESTIMATES}.
\begin{proof}[Proof of Proposition~\ref{Zn}]
 Let $Y = (Y_{ij})_{1\le i,j \le n}$ be an $n\times n$ matrix 
of i.i.d. Bernoulli random variables: 
\begin{equation*}
 {\mathbb P}(Y_{ij} = 0) = \frac{1}{2} = {\mathbb P}(Y_{ij} = 1).
\end{equation*}
For $i=1,\ldots,n$ define 
\begin{equation*}
 W_i = \sum_{j=1}^n Y_{ij}.
\end{equation*}
Then
\begin{align*}
 Z_{n}(\beta_{1},\beta_{2}) &= 2^{n^2}\,{\mathbb E}\left[\exp\left(n^2(\beta_{1} e(Y) + \beta_{2} s(Y))\right)\right]\\
&= 2^{n^2}\,{\mathbb E}\left[\exp\left(\sum_{i=1}^n \beta_{1} W_i + \frac{\beta_{2}}{n^{p-1}}W_i^p\right)\right]\\
&= 2^{n^2}\,{\mathbb E}\left[\prod_{i=1}^{n}\exp\left(\beta_{1} W_i + \frac{\beta_{2}}{n^{p-1}}W_i^p\right)\right]\\
&= 2^{n^2}\,\prod_{i=1}^{n}{\mathbb E}\left[\exp\left(\beta_{1} W_i + \frac{\beta_{2}}{n^{p-1}}W_i^p\right)\right]\\
&= 2^{n^2} \left({\mathbb E}\left[\exp\left(\beta_{1} W + \frac{\beta_{2}}{n^{p-1}}W^p\right)\right]\right)^n.
\end{align*}
\end{proof}

\begin{proof}[Proof of Proposition~\ref{E}]
We will prove only the case $k = 0$, as the 
other cases are easy extensions. Observe that 
\begin{equation*}
 {\mathbb E}\left[\exp\left(\beta_{1} W + \frac{\beta_{2}}{n^{p-1}}W^p\right)\right] = 
2^{-n} \sum_{i=1}^n \binom{n}{i} \exp\left(\beta_{1}i + \frac{\beta_{2}}{n^{p-1}}i^p\right).
\end{equation*}
Using the fact that for all $n\ge 1$, 
\begin{equation*}
 n\log n - n + \frac{1}{2}\log n \le \log n! \le  n\log n - n + \frac{1}{2}\log n + 1,
\end{equation*}
we obtain
\begin{equation}\label{bin1}
 \binom{n}{i} \le \exp\left(n\left[-\frac{i}{n}\log \frac{i}{n} - 
\left(1-\frac{i}{n}\right)\log \left(1-\frac{i}{n}\right)
+ \frac{1}{2n}\log \frac{n}{i(n-i)}\ + \frac{1}{n}\right]\right).
\end{equation}
Define 
\begin{equation*}
A_n = \{i\in \{1,\ldots,n\}\,:\,i/n \in (\varepsilon, 1 - \varepsilon)\}
\end{equation*}
where $\varepsilon>0$ will be specified momentarily.
From~\eqref{bin1}, for any $\varepsilon \in (0,1)$ we have
\begin{align}\begin{split}\label{binbound}
\max_{i \in \{1,\ldots,n\}\setminus A_n} \binom{n}{i}
\exp\left(\beta_{1} i + \frac{\beta_{2}}{n^{p-1}}i^p\right) 
&\le e\left(1-\frac{1}{n}\right)^{-\frac{1}{2}} \sup_{x \in [0,1]\setminus 
(\varepsilon, 1-\varepsilon)}
e^{n\ell(x)}\\
&\le 3\sup_{x \in [0,1]\setminus 
(\varepsilon, 1-\varepsilon)}
e^{n\ell(x)}\end{split}
\end{align}
Since $\ell'(x) \to \infty$ as $x \to 0$ and $\ell'(x) \to -\infty$ as $x \to 1$, 
the optimizer $x_*$ is in $(0,1)$ and 
we may choose $\varepsilon > 0$ such that for some $\delta > 0$,
\begin{equation*}
 \sup_{x \in [0,1]\setminus(\varepsilon,1-\varepsilon)}\ell(x) < \ell(x^*) - \delta.
\end{equation*}
Thus, $\sup_{x \in [0,1]\setminus (\varepsilon,1-\varepsilon)} e^{n\ell(x)} \le e^{n(\ell(x^*)-\delta)}$, and 
using this with~\eqref{binbound} gives
\begin{equation*}
 \sum_{i \in \{1,\ldots,n\}\setminus A_n}\binom{n}{i}
\exp\left(\beta_{1} i + \frac{\beta_{2}}{n^{p-1}}i^p\right) 
= O\left(e^{n(\ell(x^*) - \delta)}\right).
\end{equation*}
For $i \in A_n$, Stirling's formula allows us to write 
\begin{align*}
\binom{n}{i}&= \left(1+O\left(n^{-1}\right)\right)\frac{1}{\sqrt{2\pi}}\sqrt{\frac{n}{i(n-i)}}
\\
&\qquad\qquad\qquad
\times\exp\left(n\left[-\frac{i}{n}\log\frac{i}{n}- 
\left(1-\frac{i}{n}\right)\log \left(1-\frac{i}{n}\right)\right]\right).
\end{align*}
The last two displays yield
\begin{align}\begin{split}\label{mainexp}
&{\mathbb E}\left[\exp\left(\beta_{1} W + \frac{\beta_{2}}{n^{p-1}}W^p\right)\right]
\\
&=2^{-n} \left(\sum_{i=1}^n\binom{n}{i}
\exp\left(\beta_{1} i + \frac{\beta_{2}}{n^{p-1}}i^p\right)\right)\\
&= 2^{-n} \left(O\left(e^{n(\ell(x^*) - \delta)}\right) + \sum_{i \in A_n}\binom{n}{i}
\exp\left(\beta_{1} i + \frac{\beta_{2}}{n^{p-1}}i^p\right)\right)\\
&= 2^{-n}   \left(O\left(e^{n(\ell(x^*) - \delta)}\right) + \left(1+O\left(n^{-1}\right)\right)
\frac{1}{\sqrt{2\pi n}}\sum_{i \in A_n}
\sqrt{\frac{1}{(i/n)(1-i/n)}}e^{n\ell(i/n)} \right).\end{split}
\end{align}
We will approximate the sum in~\eqref{mainexp} by 
an integral. Consider first the case off the transition curve. 
Thus, there is a unique maximizer $x^*$ of $\ell$, and 
$\ell'(x^*) = 0$, $\ell''(x^*) < 0$. 
Let $q \in (1/3,1/2)$ and define
\begin{equation*}
B_n = \{i \in \{1,\ldots,n\}\,:\,i/n \in (x^*-n^{-q},x^*+n^{-q})\}.
\end{equation*}
For any $j \in A_n$, note that 
\begin{align}\begin{split}\label{err1}
&\left|\frac{1}{n}\sqrt{\frac{1}{(j/n)(1-j/n)}}e^{n\ell(j/n)} 
- \int_{j/n}^{j/n+1/n} \sqrt{\frac{1}{x(1-x)}}e^{n\ell(x)}\,dx\right| 
\\
&\le \frac{1}{n}\max_{x,y \in [j/n,\,j/n+1/n]} \left|\sqrt{\frac{1}{x(1-x)}}e^{n\ell(x)} - 
\sqrt{\frac{1}{y(1-y)}}e^{n\ell(y)}\right| \\
&\le \frac{1}{n}\max_{x\in [j/n,\,j/n+1/n]} \sqrt{\frac{1}{x(1-x)}}
\max_{x,y \in [j/n,\,j/n+1/n]}\left|e^{n\ell(x)}- e^{n\ell(y)}\right|\\
&\quad + \frac{1}{n}e^{n\ell(x^*)}\max_{x,y \in [j/n,\,j/n+1/n]}
\left|\sqrt{\frac{1}{x(1-x)}} - \sqrt{\frac{1}{y(1-y)}}\right|\\
&= O(n^{-1})\max_{x,y \in [j/n,\,j/n+1/n]} \left|e^{n\ell(x)} - e^{n\ell(y)}\right| 
+O(n^{-2})e^{n\ell(x^*)}.\end{split}
\end{align}
Fix $j \in A_n$ and let $x, y \in [j/n,j/n+1/n]$. Note that 
for all $x$, 
\begin{equation*}
|e^x -1 | \le e^{|x|}-1.
\end{equation*}
We use this, the fact that $\ell''(x^*) < 0$, 
and the mean value theorem to write
\begin{align}\begin{split}\label{err2}
\left|e^{n\ell(x)} - e^{n\ell(y)}\right| &= e^{n\ell(x^*)}e^{n(\ell(y)-\ell(x^*))}
\left|e^{n(\ell(x)-\ell(y))} - 1\right|\\
&= e^{n\ell(x^*)}\exp\left(n\frac{\ell''(x^*)}{2}(y-x^*)^2 + n\frac{\ell'''(\xi)}{6}(y-x^*)^3\right) \\
&\quad \times\left|\exp\left(n\ell'(y)(x-y) + \frac{n\ell''(\nu)}{2}(x-y)^2\right)-1\right|\\
&= e^{n\ell(x^*)}\exp\left(n\frac{\ell''(x^*)}{2}(y-x^*)^2 + n\frac{\ell'''(\xi)}{6}(y-x^*)^3\right) \\
&\quad \times\left|\exp\left(n\ell''(\zeta)(y-x^*)(x-y) + \frac{n\ell''(\nu)}{2}(x-y)^2\right)-1\right|\\
&\le e^{n\ell(x^*)}\exp\left(n\frac{-|\ell''(x^*)|}{2}(y-x^*)^2 + n\frac{\ell'''(\xi)}{6}(y-x^*)^3\right) \\
&\quad \times\left(\exp\left(n|\ell''(\zeta)||y-x^*||x-y| + \frac{n|\ell''(\nu)|}{2}(x-y)^2\right)-1\right)\end{split}
\end{align}
where $\xi$ and $\zeta$ are between $y$ and $x^*$, and 
$\nu$ is between $y$ and $x$. 
Observe that 
\begin{align*}
&\exp\left(n\frac{-|\ell''(x^*)|}{2}(y-x^*)^2 + n\frac{\ell'''(\xi)}{6}(y-x^*)^3\right)\\
&= \begin{cases}
O\left(\exp\left(-\frac{|\ell''(x^*)|}{2}n^{1-2q}\right)\right)\left(1+O(n)\right), & 
j \notin B_n\\
1+O(n^{1-3q}), & j \in B_n
\end{cases}
\end{align*}
and that 
\begin{align*}
&\exp\left(n|\ell''(\zeta)||y-x^*||x-y| + \frac{n|\ell''(\nu)|}{2}(x-y)^2\right)-1\\
&= \begin{cases}
O(1), & 
j \notin B_n\\
O(n^{-q}), & j \in B_n
\end{cases}
\end{align*}
Let $t = 1-2q > 0$ 
and $\omega \in(0, |\ell''(x^*)|/2)$. 
The last three displays show that
\begin{equation*}
 \max_{x,y \in [j/n,\,j/n+1/n]}\left|e^{n\ell(x)} - e^{n\ell(y)}\right| 
=  \begin{cases} 
e^{n\ell(x^*)}O(\exp(-\omega n^t)), & j \notin B_n \\
e^{n\ell(x^*)}O(n^{-q}), 
& j \in B_n \end{cases},
\end{equation*}
and so from~\eqref{err1}, 
\begin{align}\begin{split}\label{diff}
 &\left|\frac{1}{n}\sqrt{\frac{1}{(j/n)(1-j/n)}}e^{n\ell(j/n)} 
- \int_{j/n}^{j/n+1/n} \sqrt{\frac{1}{x(1-x)}}e^{n\ell(x)}\,dx\right| \\
&=  \begin{cases} 
e^{n\ell(x^*)}O(\exp(-\omega n^t)), & j \notin B_n \\
e^{n\ell(x^*)}O(n^{-1-q}), 
& j \in B_n \end{cases}.\end{split}
\end{align}
Observe that 
\begin{equation}\label{AnBn}
|B_n| = O(n^{1-q}),\quad |A_n\setminus B_n| = O(n).
\end{equation}
Now from~\eqref{diff}, for any $r < 1$, 
\begin{align}\begin{split}\label{sumbound}
&\left|\frac{1}{n}\sum_{i \in A_n}
\sqrt{\frac{1}{(i/n)(1-i/n)}}e^{n\ell(i/n)} - {\int_0^{1} 
\sqrt{\frac{1}{x(1-x)}}e^{n\ell(x)}dx}\right| \\
&\le e^{n\ell(x^*)}\left(|B_n|\,O(n^{-1-q}) + 
|A_n\setminus B_n|\,O(\exp(-\omega n^{t}))\right)
\\
&\qquad\qquad\qquad\qquad+ {\int_{[0,1]\backslash[\varepsilon+1/n,1-\varepsilon-1/n]} 
\sqrt{\frac{1}{x(1-x)}}e^{n\ell(x)}dx}\\
&\le e^{n\ell(x^*)}\left(O(n^{-2q}) + O(n\exp(-\omega n^{t}))\right) + O\left(e^{n(\ell(x^*) - \delta)}\right)\\
&\le e^{n\ell(x^*)}O(n^{-r}).\end{split}
\end{align}
Now by~\eqref{sumbound} and Proposition~\ref{laplace},
\begin{align*}
&\left|\frac{1}{n}\sum_{i \in A_n}
\sqrt{\frac{1}{(i/n)(1-i/n)}}e^{n\ell(i/n)} - {\int_0^{1} 
\sqrt{\frac{1}{x(1-x)}}e^{n\ell(x)}dx}\right|\\
&\quad \times\left( \int_{0}^{1} 
\sqrt{\frac{1}{x(1-x)}}e^{n\ell(x)}\,dx\right)^{-1} = O(n^{1/2-r}).
\end{align*}
Thus,
\begin{equation*}
\frac{1}{n} \sum_{i \in A_n}
\sqrt{\frac{1}{(i/n)(1-i/n)}}e^{n\ell(i/n)} 
=  \left(1+O\left(n^{1/2-r}\right)\right)\int_{0}^{1} 
\sqrt{\frac{1}{x(1-x)}}e^{n\ell(x)}\,dx.
\end{equation*}
Now from~\eqref{mainexp} we conclude
\begin{align*}
&{\mathbb E}\left[\exp\left(\beta_{1} W + \frac{\beta_{2}}{n^{p-1}}W^p\right)\right]\\
&=  \left(1+O\left(n^{1/2-r}\right)\right)
2^{-n}\sqrt{\frac{n}{2\pi}}\int_{0}^{1} 
\sqrt{\frac{1}{x(1-x)}}e^{n\ell(x)}\,dx.
\end{align*}

Next, consider $(\beta_1,\beta_2)$ 
on the transition curve away from the critical point. 
By Theorem~\ref{trans_curve}, 
there are two maximizers of $\ell$, say $x_1^*$ and $x_2^*$. 
Defining
\begin{equation*}
B_n = \{i \in \{1,\ldots,n\}\,:\,i/n \in (x_1^*-n^{-q},x_1^*+n^{-q})\cup (x_2^*-n^{-q},x_2^*+n^{-q})\}, 
\end{equation*}
it is not hard to see that the arguments above can be 
repeated to obtain the same result. 

Finally, consider the case at the critical point. 
Here, equation~\eqref{err1} still holds, but~\eqref{err2} 
needs to be modified, as follows. 
By Proposition~\ref{order}, we have $\ell'(x^*) = \ell''(x^*) = \ell'''(x^*) = 0$ 
and $\ell^{(4)}(x^*) < 0$, so by the mean value theorem we have 
\begin{align}\begin{split}
&\left|e^{n\ell(x)} - e^{n\ell(y)}\right| \\
&= e^{n\ell(x^*)}e^{n(\ell(y)-\ell(x^*))}
\left|e^{n(\ell(x)-\ell(y))} - 1\right|\\
&= e^{n\ell(x^*)}\exp\left(n\frac{\ell^{(4)}(x^*)}{4!}(y-x^*)^4 + n\frac{\ell^{(5)}(\xi)}{5!}(y-x^*)^5\right) \\
&\quad \times\left|\exp\left(n\ell'(y)(x-y) + \frac{n\ell''(\nu)}{2}(x-y)^2\right)-1\right|\\
&= e^{n\ell(x^*)}\exp\left(n\frac{\ell^{(4)}(x^*)}{4!}(y-x^*)^4 + n\frac{\ell^{(5)}(\xi)}{5!}(y-x^*)^5\right) \\
&\quad \times\left|\exp\left(n\ell^{(4)}(\zeta)(v-x^*)(u-x^*)(y-x^*)(x-y) + \frac{n\ell''(\nu)}{2}(x-y)^2\right)-1\right|\\
&\le e^{n\ell(x^*)}\exp\left(n\frac{-|\ell^{(4)}(x^*)|}{4!}(y-x^*)^4 + n\frac{\ell^{(5)}(\xi)}{5!}(y-x^*)^5\right) \\
&\quad \times\left(\exp\left(n|\ell^{(4)}(\zeta)||v-x^*||u-x^*||y-x^*||x-y| + \frac{n|\ell''(\nu)|}{2}(x-y)^2\right)-1\right)\end{split}
\end{align}
where $u$, $v$, $\xi$ and $\zeta$ are between $y$ and $x^*$, and 
$\nu$ is between $y$ and $x$. 
Let $B_{n}$ be defined as in the analysis of the case off
the critical curve.
Let $q \in (1/5,1/4)$ and note that 
\begin{align*}
&\exp\left(n\frac{-|\ell^{(4)}(x^*)|}{4!}(y-x^*)^4 + n\frac{\ell^{(5)}(\xi)}{5!}(y-x^*)^5\right)\\
&= \begin{cases}
O\left(\exp\left(\frac{-|\ell^{(4)}(x^*)|}{4!}n^{1-4q}\right)\right)
\left(1 + O(n)\right), & j \notin B_n\\ 
1 + O(n^{1-5q}), & j \in B_n
\end{cases}
\end{align*}
and that 
\begin{align*}
&\exp\left(n|\ell^{(4)}(\zeta)||v-x^*||u-x^*||y-x^*||x-y| + \frac{n|\ell''(\nu)|}{2}(x-y)^2\right)-1\\
&= \begin{cases} O(1),& j \notin B_n \\
O(n^{-3q}),& j \in B_n.\end{cases}
\end{align*}
Let $\omega \in (0, |\ell^{(4)}(x^*)|/4!)$ and 
$t = 1-4q > 0$. The last three displays show that
\begin{equation*}
 \max_{x,y \in [j/n,\,j/n+1/n]}\left|e^{n\ell(x)} - e^{n\ell(y)}\right| 
=  \begin{cases} 
e^{n\ell(x^*)}O(\exp(-\omega n^t)), & j \notin B_n \\
e^{n\ell(x^*)}O(n^{-3q}), 
& j \in B_n \end{cases}. 
\end{equation*}
So from~\eqref{err1},
\begin{align}\begin{split}\label{diff2}
 &\left|\frac{1}{n}\sqrt{\frac{1}{(j/n)(1-j/n)}}e^{n\ell(j/n)} 
- \int_{j/n}^{j/n+1/n} \sqrt{\frac{1}{x(1-x)}}e^{n\ell(x)}\,dx\right| \\
&=  \begin{cases} 
e^{n\ell(x^*)}O(\exp(-\omega n^t)), & j \notin B_n \\
e^{n\ell(x^*)}O(n^{-1-3q}), 
& j \in B_n \end{cases}.\end{split}
\end{align}
Using~\eqref{diff2} and~\eqref{AnBn},  
for any $r < 1$,
\begin{align}\begin{split}\label{sumbound2}
&\left|\frac{1}{n}\sum_{i \in A_n}
\sqrt{\frac{1}{(i/n)(1-i/n)}}e^{n\ell(i/n)} - {\int_0^{1} 
\sqrt{\frac{1}{x(1-x)}}e^{n\ell(x)}dx}\right| \\
&\le e^{n\ell(x^*)}\left(|B_n|\,O(n^{-1-3q}) + 
|A_n\setminus B_n|\,O(\exp(-\omega n^{t}))\right)
\\
&\qquad\qquad\qquad\qquad+ {\int_{[0,1]\backslash[\varepsilon+1/n,1-\varepsilon-1/n]} 
\sqrt{\frac{1}{x(1-x)}}e^{n\ell(x)}dx}\\
&\le e^{n\ell(x^*)}\left(O(n^{-4q}) + O(\exp(-\omega n^{t}))\right) + O\left(e^{n(\ell(x^*) - \delta)}\right)\\
&\le e^{n\ell(x^*)}O(n^{-r}).\end{split}
\end{align}
Now by~\eqref{sumbound2} and Proposition~\ref{laplace},
\begin{align*}
&\left|\frac{1}{n}\sum_{i \in A_n}
\sqrt{\frac{1}{(i/n)(1-i/n)}}e^{n\ell(i/n)} - {\int_0^{1} 
\sqrt{\frac{1}{x(1-x)}}e^{n\ell(x)}dx}\right|\\
&\quad \times\left( \int_{0}^{1} 
\sqrt{\frac{1}{x(1-x)}}e^{n\ell(x)}\,dx\right)^{-1} = O(n^{1/4-r}).
\end{align*}
Thus, 
\begin{equation*}
\frac{1}{n} \sum_{i \in A_n}
\sqrt{\frac{1}{(i/n)(1-i/n)}}e^{n\ell(i/n)} 
=  \left(1+O\left(n^{1/4-r}\right)\right)\int_{0}^{1} 
\sqrt{\frac{1}{x(1-x)}}e^{n\ell(x)}\,dx.
\end{equation*}
Now from~\eqref{mainexp} we conclude 
\begin{align*}
&{\mathbb E}\left[\exp\left(\beta_{1} W + \frac{\beta_{2}}{n^{p-1}}W^p\right)\right]
\\
&=\left(1+O\left(n^{1/4-r}\right)\right)
2^{-n}\sqrt{\frac{n}{2\pi}}\int_{0}^{1} 
\sqrt{\frac{1}{x(1-x)}}e^{n\ell(x)}\,dx.
\end{align*}
\end{proof}

\begin{proof}[Proof of Proposition~\ref{laplace}]
We will prove only {(i)} and {(iii)}, as {(ii)} is 
standard. We first consider {(i)}. 
Note that for $b>0$ and $k \in {\mathbb N}$, 
\begin{equation*}
 \int_{-\infty}^{\infty} x^k e^{-bx^2}\,dx = \begin{cases} 0,& k \hbox{ odd }\\ 
\Gamma\left(\frac{k+1}{2}\right)b^{-\frac{k+1}{2}},
& k \hbox{ even }\end{cases}
\end{equation*}
So for any $\delta > 0$,  
\begin{align*}
\int_{-\delta}^{\delta} u^k e^{-nbu^2} \,du
&= n^{-\frac{k+1}{2}} \int_{-\delta n^{1/2}}^{\delta n^{1/2}} x^k e^{-bx^2}\,dx \\ 
 &= n^{-\frac{k+1}{2}}\left(O(e^{-bn}) + \int_{-\infty}^\infty x^k e^{-bx^2}\,dx\right) \\
&= \begin{cases} 0,& k \hbox{ odd }\\
\Gamma\left(\frac{k+1}{2}\right)(nb)^{-\frac{k+1}{2}} 
+ O(e^{-bn}) ,
& k \hbox{ even }.\end{cases}
\end{align*}
Now let $c = x^*$ and $u = x-c$, and 
pick $0<\delta < \min\{c,1-c\}$. We use Taylor expansions of  
$x^k$ and $\ell(x)$ 
at $c$ and of $e^x$ at zero, 
along with Proposition~\ref{order}, 
to compute
\begin{align*}
&\int_{c-\delta}^{c+\delta} f(x) e^{n\ell(x)}\,dx \\
&= \int_{-\delta}^{\delta}\left[d_0 + d_1u +\ldots\right]
e^{n(b_0 + b_1u +b_2 u^2 + \ldots)}\,du\\
&= e^{n\ell(c)}\int_{-\delta}^{\delta}\left[d_0 + d_1u +\ldots\right]
e^{nb_2 u^2 + nb_3u^3+ \ldots}\,du\\
&=e^{n\ell(c)}\int_{-\delta}^\delta\left[d_0 + d_1u +\ldots\right]
\left[1 + (nb_3u^3 + \ldots)  +\frac{1}{2}(nb_3u^3 + \ldots)^2+ \ldots\right]e^{nb_2 u^2}\,du\\
&= e^{n\ell(c)}\left[n^{-1/2}d_0\alpha_1 +
 n^{-3/2}\Lambda + O(n^{-5/2})\right]
\end{align*}
where the last step is obtained by collecting terms of 
the same order, and the interchange of sum and integral 
is justified by the dominated convergence theorem. Since 
$x^* = c$ is the unique global maximizer of $\ell$, we conclude that 
for some $\varepsilon > 0$,
\begin{equation*}
\int_{0}^{1} f(x) e^{n\ell(x)}\,dx 
= \int_{c-\delta}^{c+\delta} f(x) e^{n\ell(x)}\,dx 
+ O\left(e^{n(\ell(c)-\varepsilon)}\right).
\end{equation*}
It follows that 
\begin{equation*}
\int_{0}^{1} f(x) e^{n\ell(x)}\,dx 
= e^{n\ell(c)}\left[n^{-1/2}d_0\alpha_1 +
 n^{-3/2}\Lambda + O(n^{-5/2})\right].
\end{equation*}

Now we turn to {(iii)}. 
Note that for $b>0$ and $k \in {\mathbb N}$, 
\begin{equation*}
 \int_{-\infty}^{\infty} x^k e^{-bx^4}\,dx = \begin{cases} 0,& k \hbox{ odd }\\ 
\frac{1}{2}\Gamma\left(\frac{k+1}{4}\right)b^{-\frac{k+1}{4}},
& k \hbox{ even }\end{cases}.
\end{equation*}
So for any $\delta > 0$, 
\begin{align*}
\int_{-\delta}^\delta u^k e^{-nbu^4} \,du
&= n^{-\frac{k+1}{4}} \int_{-\delta n^{1/4}}^{\delta n^{1/4}} x^k e^{-bx^4}\,dx \\ 
 &= n^{-\frac{k+1}{4}}\left(O(e^{-bn}) + \int_{-\infty}^\infty x^k e^{-bx^4}\,dx\right) \\
&= \begin{cases} 0,& k \hbox{ odd }\\
\frac{1}{2}\Gamma\left(\frac{k+1}{4}\right)(nb)^{-\frac{k+1}{4}} + O(e^{-bn}) ,
& k \hbox{ even }\end{cases}.
\end{align*}
As before we let $c = x^*$ and $u = x-c$, 
pick $0< \delta = \min\{c,1-c\}$ and use 
Taylor expansions of $x^k$ and $\ell(x)$ 
at $c$ and $e^x$ at zero, along with
 Proposition~\ref{order}, to write 
\begin{align*}
&\int_{c-\delta}^{c+\delta} f(x) e^{n\ell(x)}\,dx\\
&= \int_{c-\delta}^{c+\delta}\left[d_0 + d_1u +\ldots\right]
e^{n(b_0 + b_1u +b_2 u^2 + \ldots)}\,du\\
&= e^{n\ell(c)}\int_{c-\delta}^{c+\delta}\left[d_0 + d_1u +\ldots\right]
e^{nb_4 u^4 + nb_5u^5+ \ldots}\,du\\
&=e^{n\ell(c)}\int_{c-\delta}^{c+\delta}\left[d_0 + d_1u +\ldots\right]
\left[1 + (nb_5u^5 + \ldots) +\frac{1}{2}(nb_5u^5 + \ldots)^2 + \ldots\right]e^{nb_4 u^4}\,du\\
&= e^{n\ell(c)}\left[n^{-1/4}d_0\gamma_1 +
 n^{-3/4}\Theta + O(n^{-5/4})\right],
\end{align*}
where again the last step is obtained by collecting terms of 
the same order, and the interchange of sum and integral 
is justified by the dominated convergence theorem. As 
before, since $x^* = c$ is the unique global maximizer of $\ell$, 
we can conclude that 
\begin{equation*}
\int_{0}^{1} f(x) e^{n\ell(x)}\,dx = 
e^{n\ell(c)}\left[n^{-1/4}d_0\gamma_1 +
 n^{-3/4}\Theta + O(n^{-5/4})\right].
\end{equation*}
\end{proof}

The remainder of the proofs are for the 
results in Section~\ref{THEOREMS}.
\begin{proof}[Proof of Theorem~\ref{free_energy}]
By Propositions~\ref{E} and~\ref{laplace}, we have
\begin{align}\begin{split}\label{long}
\psi_n(\beta_{1},\beta_{2}) &=  n^{-2} \log Z_n(\beta_{1},\beta_{2}) \\
&= \log 2 + n^{-1} \log {\mathbb E}
\left[\exp\left(\beta_{1} W + \frac{\beta_{2}}{n^{p-1}}W^p\right)\right] \\
&= O(n^{-1}\log n) + \frac{1}{n}\log \int_0^1 \sqrt{\frac{1}{x(1-x)}}e^{n\ell(x)}\,dx\\
&= O(n^{-1}\log n) + \ell(x^*).
\end{split}
\end{align}
\end{proof}

\begin{proof}[Proof of Theorem~\ref{PropertyCurve}]
(i) Along the phase transition curve, we have
\begin{align}
&\beta_{1}+pq(\beta_{1})(x_{1}^{\ast})^{p-1}-\log\left(\frac{x_{1}^{\ast}}{1-x_{1}^{\ast}}\right)=0,
\label{EqnI}
\\
&\beta_{1}+pq(\beta_{1})(x_{2}^{\ast})^{p-1}-\log\left(\frac{x_{2}^{\ast}}{1-x_{2}^{\ast}}\right)=0,
\label{EqnII}
\\
&\beta_{1}x_{1}^{\ast}+q(\beta_{1})(x_{1}^{\ast})^{p}
-x_{1}^{\ast}\log x_{1}^{\ast}-(1-x_{1}^{\ast})\log(1-x_{1}^{\ast})
\nonumber
\\
&\qquad\qquad
=\beta_{1}x_{2}^{\ast}+q(\beta_{1})(x_{2}^{\ast})^{p}
-x_{2}^{\ast}\log x_{2}^{\ast}-(1-x_{2}^{\ast})\log(1-x_{2}^{\ast}).
\label{EqnIII}
\end{align}
Let $x_1^*<x_2^*$ be the two local maximizers of $\ell$ in 
the V-shaped region~\cite{Radin} that contains 
the phase transition curve except the critical point. 
By Proposition~\ref{order}, $\ell''(x_1^*)$ and $\ell''(x_2^*)$ 
are nonzero away from the critical point. The implicit 
function theorem implies that then $x_{1}^{\ast}$
and $x_{2}^{\ast}$ are analytic functions of 
both $\beta_{1}$ and $\beta_2$. 
Differentiating \eqref{EqnIII} with respect to $\beta_{1}$ and using
\eqref{EqnI} and \eqref{EqnII}, we can show that
\begin{equation*}
x_{1}^{\ast}+q'(\beta_{1})(x_{1}^{\ast})^{p}
=
x_{2}^{\ast}+q'(\beta_{1})(x_{2}^{\ast})^{p},
\end{equation*}
which implies that 
\begin{equation}\label{qprime}
q'(\beta_{1})=-\frac{x_{1}^{\ast}-x_{2}^{\ast}}{(x_{1}^{\ast})^{p}-(x_{2}^{\ast})^{p}}.
\end{equation}
As $\beta_{1}\rightarrow\beta_{1}^{c}$, $x_{2}^{\ast}-x_{1}^{\ast}\rightarrow 0$
and both $x_{2}^{\ast}$ and $x_{1}^{\ast}$ converge to the common maximizer 
$x^{\ast}_{c}=\frac{p-1}{p}$. Therefore,
\begin{equation*}
\lim_{\beta_{1}\rightarrow\beta_{1}^{c}}q'(\beta_{1})
=-\frac{1}{p(x^{\ast}_{c})^{p-1}}=-\frac{p^{p-2}}{(p-1)^{p-1}}.
\end{equation*}
Since $x_{1}^{\ast}\rightarrow 0$ and $x_{2}^{\ast}\rightarrow 1$ as $\beta_{1}\rightarrow-\infty$,
we get $\lim_{\beta_{1}\rightarrow-\infty}q'(\beta_{1})=-1$.

(ii) Differentiating $q'(\beta_{1})$ with respect to $\beta_{1}$, we get
\begin{align}
q''(\beta_{1})&=
-\frac{1}{((x_{1}^{\ast})^{p}-(x_{2}^{\ast})^{p})^{2}}
\left[(1-p)(x_{1}^{\ast})^{p}
+p(x_{1}^{\ast})^{p-1}x_{2}^{\ast}-(x_{2}^{\ast})^{p}\right]
\frac{\partial x_{1}^{\ast}}{\partial\beta_{1}}
\nonumber
\\
&\qquad
-\frac{1}{((x_{1}^{\ast})^{p}-(x_{2}^{\ast})^{p})^{2}}
\left[(1-p)(x_{2}^{\ast})^{p}
+p(x_{2}^{\ast})^{p-1}x_{1}^{\ast}-(x_{1}^{\ast})^{p}\right]
\frac{\partial x_{2}^{\ast}}{\partial\beta_{1}}.\label{SecondDerivative}
\end{align}
Differentiating \eqref{EqnI} and \eqref{EqnII} with respect to $\beta_{1}$, we get
\begin{align}
&1+pq'(\beta_{1})(x_{1}^{\ast})^{p-1}
+\left[pq(\beta_{1})(p-1)(x_{1}^{\ast})^{p-2}-\frac{1}{x_{1}^{\ast}(1-x_{1}^{\ast})}\right]
\frac{\partial x_{1}^{\ast}}{\partial\beta_{1}}=0,\label{EqnIV}
\\
&1+pq'(\beta_{1})(x_{2}^{\ast})^{p-1}
+\left[pq(\beta_{1})(p-1)(x_{2}^{\ast})^{p-2}-\frac{1}{x_{2}^{\ast}(1-x_{2}^{\ast})}\right]
\frac{\partial x_{2}^{\ast}}{\partial\beta_{1}}=0.\label{EqnV}
\end{align}
Notice that, from~\eqref{qprime},
\begin{align}\begin{split}\label{EqnVI}
1+pq'(\beta_{1})(x_{1}^{\ast})^{p-1}
&=1-p\frac{x_{1}^{\ast}-x_{2}^{\ast}}{(x_{1}^{\ast})^{p}-(x_{2}^{\ast})^{p}}(x_{1}^{\ast})^{p-1} \\
&= 1 - \frac{p(x_1^*)^{p-1}}{(x_1^*)^{p-1}+(x_1^*)^{p-2}x_2^*+\ldots + x_1^*(x_2^*)^{p-2}+(x_2^*)^{p-1}} \\
&>  1 - \frac{p(x_1^*)^{p-1}}{(x_1^*)^{p-1}+(x_1^*)^{p-2}x_1^*+\ldots + x_1^*(x_1^*)^{p-2}+(x_1^*)^{p-1}}\\
&=0,
\end{split}
\end{align}
and analogously
\begin{equation}\label{EqnVII}
1+pq'(\beta_{1})(x_{2}^{\ast})^{p-1}
=1-p\frac{x_{1}^{\ast}-x_{2}^{\ast}}{(x_{1}^{\ast})^{p}-(x_{2}^{\ast})^{p}}(x_{2}^{\ast})^{p-1}
<0.
\end{equation}
Moreover, in Proposition \ref{order}, we showed that
\begin{align}
&\ell''(x_{1}^{\ast})
=pq(\beta_{1})(p-1)(x_{1}^{\ast})^{p-2}-\frac{1}{x_{1}^{\ast}(1-x_{1}^{\ast})}<0,\label{EqnVIII}
\\
&
\ell''(x_{2}^{\ast})
=pq(\beta_{1})(p-1)(x_{2}^{\ast})^{p-2}-\frac{1}{x_{2}^{\ast}(1-x_{2}^{\ast})}<0.\label{EqnIX}
\end{align}
Therefore, from \eqref{EqnIV}, \eqref{EqnV}, \eqref{EqnVI}, \eqref{EqnVII}, 
\eqref{EqnVIII} and \eqref{EqnIX}, 
we conclude that $\frac{\partial x_{1}^{\ast}}{\partial\beta_{1}}>0$
and $\frac{\partial x_{2}^{\ast}}{\partial\beta_{1}}<0$.
Finally, by noticing that in \eqref{SecondDerivative},
\begin{align*}
&(1-p)(x_{1}^{\ast})^{p}
+p(x_{1}^{\ast})^{p-1}x_{2}^{\ast}-(x_{2}^{\ast})^{p}<0,
\\
&(1-p)(x_{2}^{\ast})^{p}
+p(x_{2}^{\ast})^{p-1}x_{1}^{\ast}-(x_{1}^{\ast})^{p}>0,
\end{align*}
we conclude that $q''(\beta_{1})>0$.
\end{proof}

In the proofs below, let 
$d_m^{(n)}$ be defined as in Proposition~\ref{laplace} 
for the function 
\begin{equation*}
f(x) = \frac{x^n}{\sqrt{x(1-x)}}.
\end{equation*}

\begin{proof}[Proof of Theorem~\ref{MainThm}]
Off the phase transition curve, the result 
follows immediately from Theorem~\ref{free_energy} 
and results in~\cite{Radin}. 
Thus, we prove only the last two displays in 
Theorem~\ref{MainThm}.

>From the second line of~\eqref{long}, we have
\begin{align}\label{2nd_deriv}
 \frac{\partial^2}{\partial \beta_{1}^2}\psi_n(\beta_{1},\beta_{2}) 
&= n^{-1}\Bigg\{\frac{{\mathbb E}\left[W^2\exp\left(\beta_{1} W + \frac{\beta_{2}}{n^{p-1}}W^p\right)\right]}
{{\mathbb E}\left[\exp\left(\beta_{1} W + \frac{\beta_{2}}{n^{p-1}}W^p\right)\right]}
\\
&\qquad\qquad\qquad
- \left(\frac{{\mathbb E}\left[W\exp\left(\beta_{1} W + \frac{\beta_{2}}{n^{p-1}}W^p\right)\right]}
{{\mathbb E}\left[\exp\left(\beta_{1} W + \frac{\beta_{2}}{n^{p-1}}W^p\right)\right]}\right)^2\Bigg\}.
\nonumber
\end{align}
We use Proposition~\ref{E} and Proposition~\ref{laplace} 
to estimate each of the terms in~\eqref{2nd_deriv}.

We first consider the case on the transition curve 
excluding the critical point. By Theorem~\ref{trans_curve}, there are
two global maximizers $x_{1}^{\ast}< x_{2}^{\ast}$ of $\ell$. 
Let us write $\ell(x_{1}^{\ast})=\ell(x_{2}^{\ast})=\ell(x^{\ast})$.
By Proposition \ref{laplace} and Proposition \ref{E}, for any $r<1$, we have
\begin{align}\label{Moments}
&{\mathbb E}\left[W^{k}\exp\left(\beta_{1} W + \frac{\beta_{2}}{n^{p-1}}W^p\right)\right]
\\
&=\left[1+O(n^{\frac{1}{2}-r})\right]
\frac{n^{k}2^{-n}\sqrt{n}}{\sqrt{2\pi}}\int_{0}^{1}\sqrt{\frac{x^{2k}}{x(1-x)}}e^{n\ell(x)}dx
\nonumber
\\
&=\left[1+O(n^{\frac{1}{2}-r})\right]
\frac{n^{k}2^{-n}\sqrt{n}}{\sqrt{2\pi}}\frac{e^{n\ell(x^{\ast})}}{\sqrt{n}}
\left[\frac{\sqrt{\frac{(x_{1}^{\ast})^{2k}}{x_{1}^{\ast}(1-x_{1}^{\ast})}}}
{\sqrt{2\pi\ell''(x_{1}^{\ast})}}
+\frac{\sqrt{\frac{(x_{2}^{\ast})^{2k}}{x_{2}^{\ast}(1-x_{2}^{\ast})}}}
{\sqrt{2\pi\ell''(x_{2}^{\ast})}}+O(n^{-1})\right]
\nonumber
\\
&=
\frac{n^{k}2^{-n}e^{n\ell(x^{\ast})}}{2\pi}
\left[\frac{(x_{1}^{\ast})^{k}}
{\sqrt{x_{1}^{\ast}(1-x_{1}^{\ast})\ell''(x_{1}^{\ast})}}
+\frac{(x_{2}^{\ast})^{k}}
{\sqrt{x_{2}^{\ast}(1-x_{2}^{\ast})\ell''(x_{2}^{\ast})}}+O(n^{\frac{1}{2}-r})\right].
\nonumber
\end{align}
Hence,
\begin{align}
&\frac{\partial^{2}}{\partial\beta_{1}^{2}}\psi_{n}(\beta_{1},\beta_{2})
\\
&=n^{-1}n^{2}\frac{\frac{(x_{1}^{\ast})^{2}}
{\sqrt{x_{1}^{\ast}(1-x_{1}^{\ast})|\ell''(x_{1}^{\ast})|}}
+\frac{(x_{2}^{\ast})^{2}}
{\sqrt{x_{2}^{\ast}(1-x_{2}^{\ast})|\ell''(x_{2}^{\ast})|}}}
{\frac{1}
{\sqrt{x_{1}^{\ast}(1-x_{1}^{\ast})|\ell''(x_{1}^{\ast})|}}
+\frac{1}
{\sqrt{x_{2}^{\ast}(1-x_{2}^{\ast})|\ell''(x_{2}^{\ast})|}}}
\nonumber
\\
&\qquad\qquad\qquad\qquad
-n^{-1}n^{2}
\frac{\left(\frac{x_{1}^{\ast}}
{\sqrt{x_{1}^{\ast}(1-x_{1}^{\ast})|\ell''(x_{1}^{\ast})|}}
+\frac{x_{2}^{\ast}}
{\sqrt{x_{2}^{\ast}(1-x_{2}^{\ast})|\ell''(x_{2}^{\ast})|}}\right)^{2}}
{\left(\frac{1}
{\sqrt{x_{1}^{\ast}(1-x_{1}^{\ast})|\ell''(x_{1}^{\ast})|}}
+\frac{1}
{\sqrt{x_{2}^{\ast}(1-x_{2}^{\ast})|\ell''(x_{2}^{\ast})|}}\right)^{2}}+O(n^{\frac{3}{2}-r})
\nonumber
\\
&=
n
\frac{\frac{(x_{1}^{\ast}-x_{2}^{\ast})^{2}}{\sqrt{x_{1}^{\ast}(1-x_{1}^{\ast})|\ell''(x_{1}^{\ast})|}
\sqrt{x_{2}^{\ast}(1-x_{2}^{\ast})|\ell''(x_{2}^{\ast})|}}}
{\left(\frac{1}
{\sqrt{x_{1}^{\ast}(1-x_{1}^{\ast})|\ell''(x_{1}^{\ast})|}}
+\frac{1}
{\sqrt{x_{2}^{\ast}(1-x_{2}^{\ast})|\ell''(x_{2}^{\ast})|}}\right)^{2}}
+O(n^{\frac{3}{2}-r})
\nonumber
\\
&=
n
\frac{(x_{1}^{\ast}-x_{2}^{\ast})^{2}\sqrt{x_{1}^{\ast}(1-x_{1}^{\ast})|\ell''(x_{1}^{\ast})|}
\sqrt{x_{2}^{\ast}(1-x_{2}^{\ast})|\ell''(x_{2}^{\ast})|}}
{\left(\sqrt{x_{1}^{\ast}(1-x_{1}^{\ast})|\ell''(x_{1}^{\ast})|}
+\sqrt{x_{2}^{\ast}(1-x_{2}^{\ast})|\ell''(x_{2}^{\ast})|}\right)^{2}}
+O(n^{\frac{3}{2}-r}).
\nonumber
\end{align}
Next consider the case at the critical point. 
By Proposition \ref{laplace} and Proposition \ref{E}, 
for any $r<1$,
\begin{align}
&{\mathbb E}\left[W^{k}\exp\left(\beta_{1} W + \frac{\beta_{2}}{n^{p-1}}W^p\right)\right]
\\
&=\left[1+O(n^{\frac{1}{4}-r})\right]
\frac{n^{k}2^{-n}\sqrt{n}}{\sqrt{2\pi}}\int_{0}^{1}\sqrt{\frac{x^{2k}}{x(1-x)}}e^{n\ell(x)}dx
\nonumber
\\
&=
\frac{n^{k}2^{-n}\sqrt{n}}{\sqrt{2\pi}}e^{n\ell(x^{\ast})}\left[n^{-1/4}d_{0}^{(k)}\gamma_1 +
n^{-3/4}\Theta^{(k)} + O(n^{-r})\right],
\nonumber
\end{align}
where 
\begin{equation*}
\Theta^{(k)}:=d_{2}^{(k)}\gamma_3 + d_{1}^{(k)} b_5\gamma_7 + 
d_{0}^{(k)} b_6 \gamma_7 + \frac{1}{2}d_{0}^{(k)} b_5^2 \gamma_{11},
\qquad\qquad k=0,1,2.
\end{equation*}
Then
\begin{equation*}
d_{0}^{(0)}=\frac{1}{\sqrt{x^{\ast}(1-x^{\ast})}},
\qquad
d_{0}^{(1)}=\frac{x^{\ast}}{\sqrt{x^{\ast}(1-x^{\ast})}},
\qquad
d_{0}^{(2)}=\frac{(x^{\ast})^{2}}{\sqrt{x^{\ast}(1-x^{\ast})}},
\end{equation*}
\begin{equation*}
d_{1}^{(0)}=\frac{x^{\ast}-\frac{1}{2}}{(x^{\ast}(1-x^{\ast}))^{3/2}},
\qquad
d_{1}^{(1)}=\frac{\frac{x^{\ast}}{2}}{(x^{\ast}(1-x^{\ast}))^{3/2}},
\qquad
d_{1}^{(2)}=\frac{\frac{3}{2}(x^{\ast})^{2}-(x^{\ast})^{3}}{(x^{\ast}(1-x^{\ast}))^{3/2}},
\end{equation*}
\begin{equation*}
d_{2}^{(0)}=\frac{2(x^{\ast})^{2}-2x^{\ast}+\frac{3}{4}}{2(x^{\ast}(1-x^{\ast}))^{5/2}},
\quad
d_{2}^{(1)}=\frac{(x^{\ast})^{2}-\frac{x^{\ast}}{4}}{2(x^{\ast}(1-x^{\ast}))^{5/2}},
\quad
d_{2}^{(2)}=\frac{\frac{3}{4}(x^{\ast})^{2}}{2(x^{\ast}(1-x^{\ast}))^{5/2}}.
\end{equation*}
It is easy to observe that $d_{0}^{(2)}d_{0}^{(0)}=(d_{0}^{(1)})^{2}$. By differentiating
this identity, we get $d_{1}^{(2)}d_{0}^{(0)}+d_{0}^{(2)}d_{1}^{(0)}=2d_{1}^{(1)}d_{0}^{(1)}$.
Therefore, by \eqref{2nd_deriv} and \eqref{Moments},
\begin{align}
&\frac{\partial^{2}}{\partial\beta_{1}^{2}}\psi_{n}(\beta_{1},\beta_{2})
\\
&=n^{-1}n^{2}\frac{n^{-\frac{1}{4}}d_{0}^{(2)}\gamma_{1}+n^{-\frac{3}{4}}\Theta^{(2)}}
{n^{-\frac{1}{4}}d_{0}^{(0)}\gamma_{1}+n^{-\frac{3}{4}}\Theta^{(0)}}
\nonumber
\\
&\qquad\qquad\qquad\qquad
-n^{-1}n^{2}\frac{(n^{-\frac{1}{4}}d_{0}^{(1)}\gamma_{1}+n^{-\frac{3}{4}}\Theta^{(1)})^{2}}
{(n^{-\frac{1}{4}}d_{0}^{(0)}\gamma_{1}+n^{-\frac{3}{4}}\Theta^{(0)})^{2}}
+O(n^{\frac{5}{4}-r})
\nonumber
\\
&=n\frac{n^{-1}\gamma_{1}[d_{0}^{(2)}\Theta^{(0)}+d_{0}^{(0)}\Theta^{(2)}
-2d_{0}^{(1)}\Theta^{(1)}]+O(n^{-\frac{3}{2}})}
{n^{-\frac{1}{2}}(d_{0}^{(0)})^{2}\gamma_{1}^{2}}+O(n^{\frac{5}{4}-r})
\nonumber
\\
&=\frac{n^{\frac{1}{2}}}{(d_{0}^{(0)})^{2}\gamma_{1}}
\left[\gamma_{3}\left(d_{0}^{(2)}d_{2}^{(0)}+d_{0}^{(0)}d_{2}^{(2)}
-2d_{0}^{(1)}d_{2}^{(1)}\right)\right]
\nonumber
\\
&\qquad\qquad
+\frac{n^{\frac{1}{2}}}{(d_{0}^{(0)})^{2}\gamma_{1}}
\left[b_{5}\gamma_{7}\left(d_{0}^{(2)}d_{1}^{(0)}+d_{0}^{(0)}d_{1}^{(2)}
-2d_{0}^{(1)}d_{1}^{(1)}\right)\right]
+O(n^{\frac{5}{4}-r})
\nonumber
\\
&=\frac{n^{\frac{1}{2}}\gamma_{3}}{(d_{0}^{(0)})^{2}\gamma_{1}}
\left(d_{0}^{(2)}d_{2}^{(0)}+d_{0}^{(0)}d_{2}^{(2)}
-2d_{0}^{(1)}d_{2}^{(1)}\right)
+O(n^{\frac{5}{4}-r})
\nonumber
\\
&=n^{\frac{1}{2}}\frac{\gamma_{3}}{\gamma_{1}}
+O(n^{\frac{5}{4}-r})
\nonumber
\\
&=n^{\frac{1}{2}}\frac{\Gamma(\frac{3}{4})}{\Gamma(\frac{1}{4})}
\frac{1}{\sqrt{\frac{\ell^{(4)}(x^{\ast})}{4!}}}
+O(n^{\frac{5}{4}-r})
\nonumber
=n^{\frac{1}{2}}\frac{\Gamma(\frac{3}{4})}{\Gamma(\frac{1}{4})}
\frac{2\sqrt{6}(p-1)}{p^{5/2}}
+O(n^{\frac{5}{4}-r}),
\nonumber
\end{align}
where we used Proposition \ref{order} in the last line.
\end{proof}

\begin{proof}[Proof of Theorem~\ref{starvariance}]
We prove only the last two displays in Theorem~\ref{starvariance}, 
since the first display follows immediately from 
Theorem~\ref{free_energy} and results in~\cite{Radin}.
>From the second line of~\eqref{long}, we have
\begin{align}
\frac{\partial^2}{\partial\beta_{2}^{2}}\psi_n(\beta_{1},\beta_{2})
&=n^{-1}\Bigg\{\frac{{\mathbb E}\left[\frac{W^{2p}}{n^{2(p-1)}}\exp\left(\beta_{1} W + \frac{\beta_{2}}{n^{p-1}}W^p\right)\right]}
{{\mathbb E}\left[\exp\left(\beta_{1} W + \frac{\beta_{2}}{n^{p-1}}W^p\right)\right]}
\\
&\qquad\qquad\qquad
- \left(\frac{{\mathbb E}\left[\frac{W^{p}}{n^{p-1}}\exp\left(\beta_{1} W + \frac{\beta_{2}}{n^{p-1}}W^p\right)\right]}
{{\mathbb E}\left[\exp\left(\beta_{1} W + \frac{\beta_{2}}{n^{p-1}}W^p\right)\right]}\right)^2\Bigg\}.
\nonumber
\end{align}
Consider first the case on the phase transition curve excluding the critical point. 
Then, similar to the proof of Theorem \ref{MainThm}, for any $r<1$, 
\begin{align*}
\frac{\partial^{2}}{\partial\beta_{2}^{2}}\psi_{n}(\beta_{1},\beta_{2})
&=n\frac{((x_{1}^{\ast})^{p}-(x_{2}^{\ast})^{p})^{2}\sqrt{x_{1}^{\ast}(1-x_{1}^{\ast})|\ell''(x_{1}^{\ast})|}
\sqrt{x_{2}^{\ast}(1-x_{2}^{\ast})|\ell''(x_{2}^{\ast})|}}
{\left(\sqrt{x_{1}^{\ast}(1-x_{1}^{\ast})|\ell''(x_{1}^{\ast})|}
+\sqrt{x_{2}^{\ast}(1-x_{2}^{\ast})|\ell''(x_{2}^{\ast})|}\right)^{2}}
\\
&\qquad\qquad\qquad\qquad\qquad
+O(n^{\frac{3}{2}-r}).
\end{align*}
Now consider the case at the critical point. We have 
\begin{align*}
&d_{0}^{(p)}=\frac{(x^{\ast})^{p}}{\sqrt{x^{\ast}(1-x^{\ast})}},
\\
&d_{1}^{(p)}=\frac{(p-\frac{1}{2})(x^{\ast})^{p}-(p-1)(x^{\ast})^{p+1}}{(x^{\ast}(1-x^{\ast}))^{3/2}},
\\
&d_{2}^{(p)}=\frac{(p^{2}-2p+\frac{3}{4})(x^{\ast})^{p}
-(2p^{2}-5p+2)(x^{\ast})^{p+1}
+(p^{2}-3p+2)(x^{\ast})^{p+2}}{2(x^{\ast}(1-x^{\ast}))^{5/2}}.
\end{align*}
It is easy to observe that $d_{0}^{(2p)}d_{0}^{(0)}=(d_{0}^{(p)})^{2}$. By differentiating
this identity, we get $d_{1}^{(2p)}d_{0}^{(0)}+d_{0}^{(2p)}d_{1}^{(0)}=2d_{1}^{(p)}d_{0}^{(p)}$.
Similar to the proof of Theorem \ref{MainThm}, for any $r<1$,
\begin{align}
&\frac{\partial^{2}}{\partial\beta_{1}^{2}}\psi_{n}(\beta_{1},\beta_{2})
\nonumber
\\
&=n\frac{(n^{-\frac{1}{4}}d_{0}^{(2p)}\gamma_{1}+n^{-\frac{3}{4}}\Theta^{(2p)})
(n^{-\frac{1}{4}}d_{0}^{(0)}\gamma_{1}+n^{-\frac{3}{4}}\Theta^{(0)})
-(n^{-\frac{1}{4}}d_{0}^{(p)}\gamma_{1}+n^{-\frac{3}{4}}\Theta^{(p)})^{2}}
{(n^{-\frac{1}{4}}d_{0}^{(0)}\gamma_{1}+n^{-\frac{3}{4}}\Theta^{(0)})^{2}}
\nonumber
\\
&\qquad\qquad\qquad\qquad\qquad
+O(n^{\frac{5}{4}-r})
\nonumber
\\
&=\frac{n^{\frac{1}{2}}\gamma_{3}}{(d_{0}^{(0)})^{2}\gamma_{1}}
\left(d_{0}^{(2p)}d_{2}^{(0)}+d_{0}^{(0)}d_{2}^{(2p)}
-2d_{0}^{(p)}d_{2}^{(p)}\right)
+O(n^{\frac{5}{4}-r})
\nonumber
\\
&=p^{2}(x^{\ast})^{2p-2}\frac{\gamma_{3}}{\gamma_{1}}n^{1/2}
+O(n^{\frac{5}{4}-r})
\nonumber
\\
&=n^{\frac{1}{2}}p^{2}\left(\frac{p-1}{p}\right)^{2p-2}
\frac{\Gamma(\frac{3}{4})}{\Gamma(\frac{1}{4})}
\frac{2\sqrt{6}(p-1)}{p^{5/2}}
+O(n^{\frac{5}{4}-r}).
\nonumber
\end{align}
\end{proof}

\begin{proof}[Proof of Theorem~\ref{covariance}]
Again we prove only the last two displays in the theorem.
>From the second line of~\eqref{long}, we have
\begin{align}
&\frac{\partial^2}{\partial\beta_{1}\partial\beta_{2}}\psi_n(\beta_{1},\beta_{2})
\\
&=n^{-1}\frac{{\mathbb E}\left[W\frac{W^{p}}{n^{(p-1)}}\exp\left(\beta_{1} W + \frac{\beta_{2}}{n^{p-1}}W^p\right)\right]}
{{\mathbb E}\left[\exp\left(\beta_{1} W + \frac{\beta_{2}}{n^{p-1}}W^p\right)\right]}
\nonumber
\\
&\qquad
-\frac{{{\mathbb E}\left[W\exp\left(\beta_{1} W + \frac{\beta_{2}}{n^{p-1}}W^p\right)\right]}
{\mathbb E}\left[\frac{W^{p}}{n^{p-1}}
\exp\left(\beta_{1} W + \frac{\beta_{2}}{n^{p-1}}W^p\right)\right]}
{\left({\mathbb E}\left[\exp\left(\beta_{1} W + \frac{\beta_{2}}{n^{p-1}}W^p\right)\right]\right)^{2}}.
\nonumber
\end{align}
Similar to the proof of Theorem \ref{MainThm},
on the phase transition curve excluding the critical point, for any $r<1$,
\begin{align*}
&\frac{\partial^2 }{\partial \beta_1^2}\psi_n(\beta_1,\beta_2)
\\
&=
n\frac{((x_{1}^{\ast})^{p}-(x_{2}^{\ast})^{p})(x_{1}^{\ast}-x_{2}^{\ast})
\sqrt{x_{1}^{\ast}(1-x_{1}^{\ast})|\ell''(x_{1}^{\ast})|}
\sqrt{x_{2}^{\ast}(1-x_{2}^{\ast})|\ell''(x_{2}^{\ast})|}}
{\left(\sqrt{x_{1}^{\ast}(1-x_{1}^{\ast})|\ell''(x_{1}^{\ast})|}
+\sqrt{x_{2}^{\ast}(1-x_{2}^{\ast})|\ell''(x_{2}^{\ast})|}\right)^{2}}
+O(n^{\frac{3}{2}-r}).
\end{align*}
Consider now the case at the critical point. 
It is easy to observe that $d_{0}^{(p+1)}d_{0}^{(0)}=(d_{0}^{(1)})(d_{0}^{(p)})$. By differentiating
this identity, we get $d_{1}^{(p+1)}d_{0}^{(0)}+d_{0}^{(p+1)}d_{1}^{(0)}=d_{1}^{(1)}d_{0}^{(p)}
+d_{0}^{(1)}d_{1}^{(p)}$.
Therefore, similar to the proof of Theorem \ref{MainThm}, we get for any $r<1$,
\begin{align}
&\frac{\partial^{2}}{\partial\beta_{1}^{2}}\psi_{n}(\beta_{1},\beta_{2})
\nonumber
\\
&=n\frac{(n^{-\frac{1}{4}}d_{0}^{(p+1)}\gamma_{1}+n^{-\frac{3}{4}}\Theta^{(p+1)})
(n^{-\frac{1}{4}}d_{0}^{(0)}\gamma_{1}+n^{-\frac{3}{4}}\Theta^{(0)})}
{(n^{-\frac{1}{4}}d_{0}^{(0)}\gamma_{1}+n^{-\frac{3}{4}}\Theta^{(0)})^{2}}
\\
&\qquad\qquad
-n\frac{(n^{-\frac{1}{4}}d_{0}^{(1)}\gamma_{1}+n^{-\frac{3}{4}}\Theta^{(1)})
(n^{-\frac{1}{4}}d_{0}^{(p)}\gamma_{1}+n^{-\frac{3}{4}}\Theta^{(p)})}
{(n^{-\frac{1}{4}}d_{0}^{(0)}\gamma_{1}+n^{-\frac{3}{4}}\Theta^{(0)})^{2}}
+O(n^{\frac{5}{4}-r})
\nonumber
\\
&=n\frac{n^{-1}\gamma_{1}[d_{0}^{(p+1)}\Theta^{(0)}+d_{0}^{(0)}\Theta^{(p+1)}
-d_{0}^{(1)}\Theta^{(p)}-d_{0}^{(p)}\Theta^{(1)}]+O(n^{-\frac{3}{2}})}
{n^{-\frac{1}{2}}(d_{0}^{(0)})^{2}\gamma_{1}^{2}+O(n^{-1})}
+O(n^{\frac{5}{4}-r})
\nonumber
\\
&=\frac{n^{\frac{1}{2}}\gamma_{3}}{(d_{0}^{(0)})^{2}\gamma_{1}}
\left(d_{0}^{(p+1)}d_{2}^{(0)}+d_{0}^{(0)}d_{2}^{(p+1)}
-d_{0}^{(1)}d_{2}^{(p)}-d_{0}^{(p)}d_{2}^{(1)}\right)
+O(n^{\frac{5}{4}-r})
\nonumber
\\
&=p(x^{\ast})^{p-1}\frac{\gamma_{3}}{\gamma_{1}}n^{1/2}+O(n^{\frac{5}{4}-r})
\nonumber
\\
&=p\left(\frac{p-1}{p}\right)^{p-1}\frac{\Gamma(\frac{3}{4})}{\Gamma(\frac{1}{4})}
\frac{2\sqrt{6}(p-1)}{p^{5/2}}n^{1/2}+O(n^{\frac{5}{4}-r}).
\nonumber
\end{align}
\end{proof}

\begin{proof}[Proof of Theorem~\ref{marginaldensities}]
Observe first that $\mathbb{P}_{n}(X_{12}=1)=\mathbb{E}_{n}[X_{12}]
=\frac{1}{n}\mathbb{E}_{n}[\sum_{j=1}^{n}X_{1j}]$.
Thus, off the transition curve we have
\begin{align*}
\lim_{n\rightarrow\infty}\mathbb{P}_{n}(X_{12}=1)
&=\lim_{n\rightarrow\infty}\frac{1}{n}\mathbb{E}_{n}\left[\sum_{j=1}^{n}X_{1j}\right]
\\
&=\lim_{n\rightarrow\infty}\frac{1}{n}\frac{\mathbb{E}\left[W\exp\left(\beta_{1} W + \frac{\beta_{2}}{n^{p-1}}W^p\right)\right]}
{\mathbb{E}\left[\exp\left(\beta_{1} W + \frac{\beta_{2}}{n^{p-1}}W^p\right)\right]}
\\
&=\lim_{n\rightarrow\infty}
\frac{\left(1+O\left(n^{1/2-4q}\right)\right)
2^{-n}\sqrt{\frac{n}{2\pi}}\int_{0}^{1} 
\sqrt{\frac{x^{2}}{x(1-x)}}e^{n\ell(x)}\,dx}{\left(1+O\left(n^{1/2-4q}\right)\right)
2^{-n}\sqrt{\frac{n}{2\pi}}\int_{0}^{1} 
\sqrt{\frac{1}{x(1-x)}}e^{n\ell(x)}\,dx}
\\
&=\lim_{n\rightarrow\infty}\frac{\sqrt{\frac{2\pi(x^{\ast})^{2}}
{x^{\ast}(1-x^{\ast})|\ell''(x^{\ast})|}}n^{-\frac{1}{2}}e^{n\ell(x^{\ast})}}
{\sqrt{\frac{2\pi}{x^{\ast}(1-x^{\ast})|\ell''(x^{\ast})|}}n^{-\frac{1}{2}}e^{n\ell(x^{\ast})}}
\\
&=x^{\ast}.
\end{align*}
Similarly, at the critical point, 
\begin{align*}
\lim_{n\rightarrow\infty}\mathbb{P}_{n}(X_{12}=1)
&=\lim_{n\rightarrow\infty}\frac{1}{n}\frac{\mathbb{E}\left[W\exp\left(\beta_{1} W + \frac{\beta_{2}}{n^{p-1}}W^p\right)\right]}
{\mathbb{E}\left[\exp\left(\beta_{1} W + \frac{\beta_{2}}{n^{p-1}}W^p\right)\right]}
\\
&=\lim_{n\rightarrow\infty}
\frac{\left(1+O\left(n^{1/4-4q}\right)\right)
2^{-n}\sqrt{\frac{n}{2\pi}}\int_{0}^{1} 
\sqrt{\frac{x^{2}}{x(1-x)}}e^{n\ell(x)}\,dx}{\left(1+O\left(n^{1/4-4q}\right)\right)
2^{-n}\sqrt{\frac{n}{2\pi}}\int_{0}^{1} 
\sqrt{\frac{1}{x(1-x)}}e^{n\ell(x)}\,dx}
\\
&=\lim_{n\rightarrow\infty}\frac{e^{n\ell(x^{\ast})}n^{-\frac{1}{4}}d_{0}^{(1)}\gamma_{1}}
{e^{n\ell(x^{\ast})}n^{-\frac{1}{4}}d_{0}^{(0)}\gamma_{1}}
\\
&=x^{\ast}.
\end{align*}
Finally, on the phase transition curve except at the critical point,
\begin{align*}
\lim_{n\rightarrow\infty}\mathbb{P}_{n}(X_{12}=1)
&=\lim_{n\rightarrow\infty}\frac{1}{n}\frac{\mathbb{E}\left[W\exp\left(\beta_{1} W + \frac{\beta_{2}}{n^{p-1}}W^p\right)\right]}
{\mathbb{E}\left[\exp\left(\beta_{1} W + \frac{\beta_{2}}{n^{p-1}}W^p\right)\right]}
\\
&=\lim_{n\rightarrow\infty}\frac{\left(\sqrt{\frac{2\pi(x_{1}^{\ast})^{2}}
{x_{1}^{\ast}(1-x_{1}^{\ast})|\ell''(x_{1}^{\ast})|}}
+\sqrt{\frac{2\pi(x_{2}^{\ast})^{2}}
{x_{2}^{\ast}(1-x_{2}^{\ast})|\ell''(x_{2}^{\ast})|}}\right)n^{-\frac{1}{2}}e^{n\ell(x^{\ast})}}
{\left(\sqrt{\frac{2\pi}{x_{1}^{\ast}(1-x_{1}^{\ast})|\ell''(x_{1}^{\ast})|}}
+\sqrt{\frac{2\pi}{x_{2}^{\ast}(1-x_{2}^{\ast})|\ell''(x_{2}^{\ast})|}}\right)n^{-\frac{1}{2}}e^{n\ell(x^{\ast})}}
\\
&=\frac{x_{1}^{\ast}\sqrt{\frac{1}
{x_{1}^{\ast}(1-x_{1}^{\ast})|\ell''(x_{1}^{\ast})|}}
+x_{2}^{\ast}\sqrt{\frac{1}
{x_{2}^{\ast}(1-x_{2}^{\ast})|\ell''(x_{2}^{\ast})|}}}
{\sqrt{\frac{1}{x_{1}^{\ast}(1-x_{1}^{\ast})|\ell''(x_{1}^{\ast})|}}
+\sqrt{\frac{1}{x_{2}^{\ast}(1-x_{2}^{\ast})|\ell''(x_{2}^{\ast})|}}}.
\end{align*}
\end{proof}

\begin{proof}[Proof of Theorem~\ref{marginaldensities2}]
(i) The results follow from $\lim_{n\rightarrow\infty}\mathbb{P}_{n}(X_{12}=1,X_{34}=1)
=\lim_{n\rightarrow\infty}\mathbb{P}_{n}(X_{12}=1)\lim_{n\rightarrow\infty}\mathbb{P}_{n}(X_{34}=1)$.

(ii) Note that $\mathbb{P}_{n}(X_{12}=X_{13}=1)=\mathbb{E}_{n}[X_{12}X_{13}]$, and
\begin{equation*}
\mathbb{E}_{n}[X_{12}X_{13}]
=\frac{1}{n^{2}-n}\left(\mathbb{E}_{n}\left[\left(\sum_{j=1}^{n}X_{1j}\right)^{2}\right]
-\mathbb{E}_{n}\left[\sum_{j=1}^{n}X_{1j}\right]\right).
\end{equation*}
It follows that off the transition curve we have
\begin{align*}
\lim_{n\rightarrow\infty}\mathbb{P}_{n}(X_{12}=X_{13}=1)
&=\lim_{n\rightarrow\infty}\frac{1}{n^{2}}\mathbb{E}_{n}\left[\left(\sum_{j=1}^{n}X_{1j}\right)^{2}\right]
\\
&=\lim_{n\rightarrow\infty}\frac{1}{n^{2}}\frac{\mathbb{E}[W^{2}\exp(\beta_{1}W+\frac{\beta_{2}}{n^{p-1}}W^{p})]}
{\mathbb{E}[\exp(\beta_{1}W+\frac{\beta_{2}}{n^{p-1}}W^{p})]}
\\
&=\lim_{n\rightarrow\infty}\frac{1}{n^{2}}
\frac{\frac{n^{2}2^{-n}\sqrt{n}}{\sqrt{2\pi}}\int_{0}^{1}\sqrt{\frac{x^{4}}{x(1-x)}}e^{n\ell(x)}dx}
{\frac{2^{-n}\sqrt{n}}{\sqrt{2\pi}}\int_{0}^{1}\sqrt{\frac{1}{x(1-x)}}e^{n\ell(x)}dx}=(x^{\ast})^{2}.
\end{align*}
Similarly, at the critical point, 
\begin{align*}
\lim_{n\rightarrow\infty}\mathbb{P}_{n}(X_{12}=X_{13}=1)
&=\lim_{n\rightarrow\infty}\frac{1}{n^{2}}\frac{\mathbb{E}[W^{2}\exp(\beta_{1}W+\frac{\beta_{2}}{n^{p-1}}W^{p})]}
{\mathbb{E}[\exp(\beta_{1}W+\frac{\beta_{2}}{n^{p-1}}W^{p})]}
\\
&=\lim_{n\rightarrow\infty}
\frac{\int_{0}^{1}\sqrt{\frac{x^{4}}{x(1-x)}}e^{n\ell(x)}dx}
{\int_{0}^{1}\sqrt{\frac{1}{x(1-x)}}e^{n\ell(x)}dx}
=\frac{d_{0}^{(2)}\gamma_{1}}{d_{0}^{(0)}\gamma_{1}}=(x^{\ast})^{2}.
\end{align*}

Finally, on the phase transition curve except at the critical point,
\begin{align*}
\lim_{n\rightarrow\infty}\mathbb{P}_{n}(X_{12}=X_{13}=1)
&=\lim_{n\rightarrow\infty}\frac{1}{n^{2}}\frac{\mathbb{E}[W^{2}\exp(\beta_{1}W+\frac{\beta_{2}}{n^{p-1}}W^{p})]}
{\mathbb{E}[\exp(\beta_{1}W+\frac{\beta_{2}}{n^{p-1}}W^{p})]}
\\
&=\lim_{n\rightarrow\infty}
\frac{\int_{0}^{1}\sqrt{\frac{x^{4}}{x(1-x)}}e^{n\ell(x)}dx}
{\int_{0}^{1}\sqrt{\frac{1}{x(1-x)}}e^{n\ell(x)}dx}
\\
&=\frac{\frac{(x_{1}^{\ast})^{2}}
{\sqrt{x_{1}^{\ast}(1-x_{1}^{\ast})\ell''(x_{1}^{\ast})}}
+\frac{(x_{2}^{\ast})^{2}}
{\sqrt{x_{2}^{\ast}(1-x_{2}^{\ast})\ell''(x_{2}^{\ast})}}}
{\frac{1}
{\sqrt{x_{1}^{\ast}(1-x_{1}^{\ast})\ell''(x_{1}^{\ast})}}
+\frac{1}
{\sqrt{x_{2}^{\ast}(1-x_{2}^{\ast})\ell''(x_{2}^{\ast})}}}
\\
&=\alpha(x_{1}^{\ast})^{2}+(1-\alpha)(x_{2}^{\ast})^{2}.
\end{align*}
\end{proof}

\section*{Acknowledgements}

The authors are grateful to the Editor and an anonymous referee for helpful comments and suggestions.
D. Aristoff and L. Zhu acknowledge support from the National Science Foundation via the awards
NSF-DMS-1522398 and NSF-DMS-1613164. 
The authors are very grateful to Angelo Mele, Dan Pirjol, Charles Radin, Mei Yin and S. R. S. Varadhan 
for helpful discussions.

\end{document}